\documentclass[reqno,11pt]{amsart}

\usepackage[english]{babel}
 
\usepackage{amsmath, latexsym, amsfonts, amssymb, amsthm, amscd}
\usepackage{mathrsfs}
\usepackage{url}
\usepackage{hyperref}
\usepackage{enumerate}
\usepackage[yyyymmdd,hhmmss]{datetime}

\usepackage{stmaryrd}
\usepackage{enumitem}

\usepackage{xcolor}
\definecolor{lightgray}{gray}{0.85}
\definecolor{midgray}{gray}{0.4}
 \usepackage{soul}
 
\usepackage{fancybox}

\numberwithin{equation}{section}

\theoremstyle{remark} 
\newtheorem*{rem}{Remark} 

\theoremstyle{plain} 
\newtheorem{thm}{Theorem}[section] 
\newtheorem{cor}[thm]{Corollary}
\newtheorem{prop}[thm]{Proposition}
\newtheorem{lem}[thm]{Lemma}
 
\theoremstyle{definition} 
\newtheorem{deff}[thm]{Definition}

\newtheorem{hyp}[thm]{Hypothesis}

\newtheorem{remark}[thm]{Remark}

\title{Long-term stability of interacting Hawkes processes on random graphs}

\author{Zo\'e \textsc{Agathe-Nerine}}
\address{Universit\'e Paris Cit\'e, Laboratoire MAP5 (UMR CNRS 8145), 75270 Paris, France \& FP2M, CNRS FR 2036, \tt{zoe.agathe-nerine@u-paris.fr}.
}

\date{\today}

\begin{document}
\maketitle

\begin{abstract}
We consider a population of Hawkes processes modeling the activity of $N$ interacting neurons. The neurons are regularly positioned on the segment $[0,1]$, and the connectivity between neurons is given by a random possibly diluted and inhomogeneous graph where the probability of presence of each edge depends on the spatial position of its vertices through a spatial kernel. The main result of the paper concerns the long-time stability of the
synaptic current of the population, as $N\to\infty$, in the subcritical regime in case the synaptic memory kernel is exponential, up to time horizons that are polynomial in $N$.
\end{abstract}




\section{Introduction}

\subsection{Hawkes processes in neuroscience}

In the present paper we study the large time behavior of a population of interacting and spiking neurons, as the size of the population $N$ tends to infinity.   We model the activity of a neuron by a point process where each point represents the time of a spike: $Z_{N,i}(t)$ counts the number of spikes during the time interval $[0,t]$ of the $i$th neuron of the population. Its intensity at time $t$ conditioned on the past $[0, t)$ is given by $\lambda_{N,i}(t)$, in the sense that 
$$\mathbf{P}\left( Z_{N,i} \text{ jumps between} (t,t+dt) \vert \mathcal{F}_t\right)= \lambda_{N,i}(t)dt,$$
where $\mathcal{F}_t:=\sigma\left( Z_{N,i}(s), s\leq t, 1\leq i\leq N\right)$.

For the choice of $\lambda_{N,i}$, we want to account for the dependence of the activity of a neuron on the past of the whole population : the spike of one neuron can trigger others spikes. \textit{Hawkes processes} are then a natural choice to emphasize this interdependency. A generic choice is
\begin{equation}\label{eq:def_lambda_generic}
\lambda_{N,i}(t)=\mu(t,x_i)+f\left( v(t,x_i)+\dfrac{1}{N}\sum_{j=1}^N w_{ij}^{(N)} \int_0^{t-} h(t-s) dZ_{N,j}(s)\right).
\end{equation}
Here, with the $i$th neuron at position $x_i=\frac{i}{N}\in I:=[0,1]$, $f ~:~  \mathbb{R}  \longrightarrow \mathbb{R}_+$ represents the synaptic integration, $\mu(t,\cdot)~:~ I \longrightarrow \mathbb{R}_+$ a spontaneous activity of the neuron at time $t$, $v(t,\cdot)~:~ I \longrightarrow \mathbb{R}$ a past activity and  $h~:~  \mathbb{R}_+ \longrightarrow \mathbb{R}$ a memory function which models how a past jump of the system affects the present intensity. The term $w_{ij}^{(N)}$ represents the random inhomogeneous interaction between neurons $i$ and $j$, that will be modeled here in terms of the realization of a random graph.

Since the seminal works of \cite{HAWKES1971, Hawkes1974}, there has been a renewed interest in the use of Hawkes processes, especially in neuroscience. A common simplified framework is to consider an interaction on the complete graph, that is taking $w_{ij}^{(N)}=1$ in \eqref{eq:def_lambda_generic}, as done in \cite{delattre2016}. In this case, a very simple instance of \eqref{eq:def_lambda_generic} concerns the so called \emph{linear case}, when $f(x)=x$,$\mu(t,x)=\mu$ and $v=0$, that is $\lambda_{N,i}(t)=\lambda_N(t)=\mu+\frac{1}{N}\sum_{j=1}^N \int_0^{t-} h(t-s) dZ_{N,j}(s)$, with $h\geq 0$ (see \cite{delattre2016}). The biological evidence \cite{Bosking1997,Mountcastle1997} of a spatial organisation of neurons in the brain has led to more elaborate Hawkes models with spatial interaction (see \cite{Touboul2014,Ditlevsen2017,CHEVALLIER20191}), possibly including inhibition (see \cite{Pfaffelhuber2022}). This would correspond in \eqref{eq:def_lambda_generic} to take $w_{ij}^{(N)}=W(x_i,x_j)$, where $W$ is a macroscopic interaction kernel, usual examples being the exponential distribution on $\mathbb{R}$,  $W(x,y)=\dfrac{1}{2\sigma}\exp\left( -\dfrac{\vert x-y\vert}{\sigma}\right)$ or the ``Mexican hat'' distribution
$W(x,y)=e^{-\vert x-y\vert} - Ae^{\frac{-\vert x-y\vert}{\sigma}}$, $A\in \mathbb{R},~\sigma>0$. The macroscopic limit of the multivariate Hawkes process \eqref{eq:def_lambda_generic} is then given by a family of spatially extended inhomogeneous Poisson processes whose intensities $(\lambda_t(x))_{x\in I}$ solve the convolution equation 
\begin{equation}\label{eq:def_lambda_lim_generic}
\lambda_t(x)=\mu_t(x)+f\left( v_t(x)+\int_I W(x,y) \int_0^{t} h(t-s) \lambda_s(y)dsdy\right).
\end{equation}
A crucial example is the exponential case, that is when $h(t)=e^{-\alpha t}$ for some $\alpha>0$. In this case, the Hawkes process with intensity \eqref{eq:def_lambda_generic} is Markovian (see \cite{Ditlevsen2017}). Denoting in  \eqref{eq:def_lambda_lim_generic} $u_t(x):=v_t(x)+\int_I W(x,y) \int_0^{t} h(t-s) \lambda_s(y)dsdy$ as the potential of a neuron (the synaptic current) localised in $x$ at time $t$ (so that \eqref{eq:def_lambda_lim_generic} becomes $\lambda_t(x)=f(u_t(x))$), an easy computation (see \cite{CHEVALLIER20191}) gives that, when $v_t(x)=e^{-\alpha t}v_0(x)$ for some $v_0$, $u$ solves the \emph{Neural Field Equation} (NFE) 
\begin{equation}\label{eq:NFE}
\dfrac{\partial u_t(x)}{\partial t}=-\alpha u_t(x)+\int_I W(x,y)f(u_t(y))dy+ I_t(x),
\end{equation}
with source term $I_t(x):=\int_I W(x,y)\mu_t(y)dy$. Equation \eqref{eq:NFE} has been extensively studied in the literature, mostly from a phenomenological perspective \cite{Wilson1972,Amari1977}, and is an important example of macroscopic neural dynamics with non-local interactions (we refer to \cite{Bressloff2011} for an extensive review on the subject).

In a previous work \cite{agathenerine2021multivariate}, we give a microscopic interpretation of the macroscopic kernel $W$ in terms of an inhomogeneous graph of interaction. We consider $w_{ij}^{(N)}=\xi_{ij}^{(N)} \kappa_i$ in \eqref{eq:def_lambda_generic}, where $\left(\xi_{ij}^{(N)}\right)_{1\leq i,j\leq N}$ is a collection of independent Bernoulli variables, with individual parameter $W(x_i,x_j)$: the probability that two neurons are connected depends on their spatial positions. The term $\kappa_i$ is a suitable local renormalisation parameter, to ensure that the interaction remains of order $1$. This modeling constitutes a further difficulty in the analysis as we are no longer in a mean-field framework: contrary to the case $w_{ij}^{(N)}=1$, the interaction \eqref{eq:def_lambda_generic} is no longer a functional of the empirical measure of the particles $\left(Z_{N,1},\cdots, Z_{N,N}\right)$. A recent interest has been shown to similar issues in the case of diffusions interacting on random graphs (first in the homogeneous Erd\H{o}s-R\'enyi case \cite{DelattreGL2016,Coppini2019,Coppini_Lucon_Poquet2022,Coppini2022}, and secondly for inhomogenous random graph \cite{Luon2020,bayraktar2021graphon,bet2020weakly}). 

A  common motivation between \cite{agathenerine2021multivariate} in the case of Hawkes processes and \cite{Luon2020,bayraktar2021graphon,bet2020weakly} in the case of diffusions is to understand how the inhomogeneity of the underlying graph may or may not influence the long time dynamics of the system. An issue common to all mean-field models (and their perturbations) is that there is, in general, no possibility to interchange the limits $N\to \infty	$ and $t\to\infty$. More precisely, restricting to Hawkes processes, a usual propagation of chaos result (see \cite[Theorem 8]{delattre2016}, \cite[Theorem 1]{CHEVALLIER20191},  \cite[Theorem 3.10]{agathenerine2021multivariate}) may be stated as follows: for fixed $T>0$, there exists some $C(T)>0$ such that
\begin{equation}\label{eq:chaos_generic}
\sup_{1\leq i \leq N} \mathbf{E}\left(\sup_{s\in [0,T]} \left\vert Z_{N,i}(s) - \overline{Z}_{i}(s) \right\vert \right) \leq \dfrac{C(T)}{\sqrt{N}},
\end{equation}
where $\overline{Z}_{i}$ is a Poisson process with intensity $(\lambda_t(x_i))_{t\geq 0}$ defined in \eqref{eq:def_lambda_lim_generic} suitably coupled to $Z_{N,i}$, see the above references for details. Generically, $C(T)$ is of the form $\exp(CT)$, such that \eqref{eq:chaos_generic} remains only relevant up to $T \sim c \log N$ with $c$ sufficiently small. In the pure mean-field linear case ($w_{ij}^{(N)}=1$, $f(x)=x$), there is a well known phase transition \cite[Theorems 10,11]{delattre2016} when $\Vert  h \Vert_1=\int_0^\infty h(t) dt<1$ (\emph{subcritical case}), $\lambda_t\xrightarrow[t\to\infty]{}\dfrac{\mu}{1-\Vert h \Vert_1}$, whereas when $\Vert h \Vert_1>1$ (\emph{supercritical case}), $\lambda_t\xrightarrow[t\to\infty]{}\infty$. This phase transition has been extended to the inhomogeneous case in \cite{agathenerine2021multivariate}. In the subcritical case, one can actually improve \eqref{eq:chaos_generic} in the sense that $C(T)$ is now linear in $T$ so that \eqref{eq:chaos_generic}  remains relevant up to $T=o(\sqrt{N})$. A natural question is to ask if this approximation remains valid beyond this time scale. The purpose to the present work is to address this question: we show that, in the whole generality of \eqref{eq:def_lambda_generic}, in the subcritical regime and exponential case (see details below), the macroscopic intensity \eqref{eq:def_lambda_lim_generic} converges to a finite limit when  $t\to\infty$ and that the microscopic system remains close to this limit up to polynomial times in $N$.

\subsection{Notation}
We denote by $C_{\text{parameters}}$ a constant $C>0$ which only depends on the parameters inside the lower index. These constants can change from line to line or inside a same equation, we choose just to highlight the dependency they contain. When it is not relevant, we just write $C$. For any $d\geq 1$, we denote by $\vert x\vert$ and $x \cdot y$ the Euclidean norm and scalar product of elements $x,y\in \mathbb{R}^d$. For $(E,\mathcal{A},\mu)$ a measured space, for a function $g$ in $L^p(E,\mu)$ with $p\geq 1$, we write $\Vert g \Vert_{E,\mu,p}:=\left( \int_E \vert g \vert^p d\mu \right)^\frac{1}{p}$. When $p=2$, we denote by $\langle \cdot,\cdot \rangle$ the Hermitian  scalar product in $L^2(E)$. Without ambiguity, we may omit the subscript $(E,\mu)$ or $\mu$. For a  real-valued bounded function $g$ on a space $E$,  we write $\Vert g \Vert _\infty := \Vert g \Vert _{E,\infty}=\sup_{x\in E} \vert g(x) \vert$. 

For $(E,d)$ a metric space, we denote by $ \Vert g \Vert_L = \sup_{x\neq y} \vert g(x) - g(y) \vert / d(x,y)$ the Lipschitz seminorm of a real-valued function $g$ on $E$. We denote by $\mathcal{C}(E,\mathbb{R})$ the space of continuous functions from $E$ to $\mathbb{R}$, and $\mathcal{C}_b(E,\mathbb{R})$ the space of continuous bounded ones. For any $T>0$, we denote by $\mathbb{D}\left([0,T],E\right)$ the space of c\`adl\`ag (right continuous with left limits) functions defined on $[0,T]$ and taking values in $E$. For any integer $N\geq 1$, we denote by $\llbracket 1, N \rrbracket$ the set $\left\{1,\cdots,N\right\}$. For any $p\in [0,1]$, $\mathcal{B}(p)$ denotes the Bernoulli distribution with parameter $p$.

\subsection{The model}

First, let us focus on the interaction between the particles. The graph of interaction for \eqref{eq:def_lambda_generic} is constructed as follows:

\begin{deff}\label{def:espace_proba_bb}
On a common probability space $\left(\widetilde{\Omega}, \widetilde{\mathcal{F}},\mathbb{P}\right)$, we consider a family of random variables $\xi^{(N)}=\left( \xi^{(N)}_{ij}\right)_{N\geq 1, i,j \in \llbracket 1,N \rrbracket}$ on $\widetilde{\Omega}$ such that under $\mathbb{P}$, for any $N\geq 1$ and  $i,j \in \llbracket 1,N \rrbracket$, $\xi^{(N)}$ is a collection of mutually independent Bernoulli random variables such that for $1\leq i,j \leq N$, $\xi_{ij}^{(N)}$ has parameter $W_N(\frac{i}{N},\frac{j}{N})$, where
\begin{equation}\label{eq:def_WN_P}
W_N(x,y):= \rho_N W(x,y),
\end{equation}
with $\rho_N$ some dilution parameter and $W:I^2\to [0,1]$ a macroscopic interaction kernel. We assume that the particles in \eqref{eq:def_lambda_generic} are connected according to the oriented graph $\mathcal{G}_N= \left( \left\{1,\cdots,N\right\} , \xi^{(N)}\right)$. For any $i$ and $j$, $\xi^{(N)}_{ij}=1$ encodes for the presence of the edge $j\to i$ and $\xi^{(N)}_{ij}=0$ for its absence.  The interaction in \eqref{eq:def_lambda_generic} is fixed as 
\begin{equation}\label{eq:def_wij}
w_{ij}^{(N)}=\dfrac{\xi_{ij}^{(N)}}{\rho_N},
\end{equation}
so that the interaction term remains of order 1 as $N\to\infty$.
\end{deff}

The class \eqref{eq:def_WN_P} of inhomogenous graphs falls into the framework of $W$-random graphs, see \cite{Lovsz2006,borgs2008,borgs2012}. One distinguishes the \textbf{dense case} when $\lim_{N\to\infty} \rho_N= \rho>0$ and the \textbf{diluted case} when $\rho_N \to 0$. 

We now fix these sequences, and work on a filtered probability space $\left(\Omega,\mathcal{F},\left(\mathcal{F}_t\right)_{t\geq 0},\mathbf{P}\right)$ rich enough for all the following processes can be defined. We denote by $\mathbf{E}$ the expectation under $\mathbf{P}$ and $\mathbb{E}$ the expectation w.r.t. $ \mathbb{ P}$. In the following definitions, $N$ is fixed and the particles are regularly located on the segment $I=[0,1]$. We denote by $x_i=\frac{i}{N}$ the position of the $i$-th neuron in the population of size $N$. We also divide $I$ in $N$ segments $B_{N,i}=\left(\frac{i-1}{N},\frac{i}{N}\right)$  of equal length.\\

We can now formally define our process of interest.

\begin{deff}\label{def:H2} Let $\left(\pi_i(ds,dz)\right)_{1\leq i \leq N}$ be a sequence of i.i.d. Poisson random measures on $\mathbb{R}_+\times \mathbb{R}_+$ with intensity measure $dsdz$.
A $\left(\mathcal{F}_t\right)$-adapted multivariate counting process  $\left(Z_{N,1}\left(t\right),...,Z_{N,N}\left(t\right)\right)_{t\geq 0}$ defined on $\left(\Omega,\mathcal{F},\left(\mathcal{F}_t\right)_{t\geq 0},\mathbf{P}\right)$ is called \emph{a multivariate Hawkes process} with the set of parameters
$\left(N,F,\xi^{(N)},W_N,\eta,h\right)$ if $\mathbf{P}$-almost surely, for all $t\geq 0$ and $i \in \llbracket 1, N \rrbracket$:
\begin{equation}\label{eq:def_ZiN}
Z_{N,i}(t) = \int_0^t \int_0^\infty \mathbf{1}_{\{z\leq \lambda_{N,i}(s)\}} \pi_i(ds,dz)
\end{equation}
with $\lambda_{N,i}(t)$ defined by
\begin{equation}\label{eq:def_lambdaiN_intro}
\lambda_{N,i}(t)= F(X_{N,i}(t-), \eta_t(x_i)),
\end{equation}
where
\begin{equation}\label{eq:def_UiN}
X_{N,i}(t)=\sum_{j=1}^N \dfrac{w_{ij}^{(N)}}{N}\int_0^{t} h(t-s) dZ_{N,j}(s),
\end{equation}
$\eta~:~[0, +\infty)\times I\longrightarrow \mathbb{R}^d$ for some $d \geq 1$ and $F ~:~  \mathbb{R}\times \mathbb{R}^d  \longrightarrow \mathbb{R}^+$.
\end{deff}
Our main focus is to study the quantity $\left(X_{N,i}\right)_{1\leq i \leq N}$ defined in \eqref{eq:def_UiN} as $N\to\infty$, and more precisely the random profile defined for all $x\in I$ by:
\begin{equation}\label{eq:def_UN}
X_N(t)(x):=\sum_{i=1}^N X_{N,i}(t) \mathbf{1}_{x\in\left(\frac{i-1}{N}, \frac{i}{N}\right]}.
\end{equation}

As $N \to \infty$, an informal Law of Large Numbers (LLN) argument shows that the empirical mean in \eqref{eq:def_lambdaiN_intro} becomes an expectation w.r.t. the candidate limit for $Z_{N,i}$: we can replace the sum in \eqref{eq:def_UiN} by the integral, the microscopic interaction term $w_{ij}^{(N)}$ in \eqref{eq:def_lambdaiN_intro} by the macroscopic term $W(x,y)$ (where $y$ describes the macroscopic distribution of the positions), and the past activity of the neuron $dZ_{N,j}(s)$ by its intensity in large population. In other words,  the macroscopic spatial profile will be described by 
\begin{equation}\label{eq:def_utx}
X_t(x)=\int_{I} W(x,y)\int_0^th(t-s) \lambda_s(y)ds~ dy,
\end{equation}
where the macroscopic intensity  of a neuron at position $x\in I$ denoted by $\lambda_t(x)=F(X_t(x),\eta_t(x))$ solves 
\begin{equation}\label{eq:def_lambdabarre}
\lambda_t(x)=F\left(\int_{I} W(x,y)\int_0^t h(t-s) \lambda_s(y)dsdy,\eta_t(x)\right).
\end{equation}
Such informal law of large number on a bounded time interval has been made rigorous under various settings, we refer for further references to \cite{delattre2016,CHEVALLIER20191} and more especially to \cite{agathenerine2021multivariate} which exactly incorporates the present hypotheses.
\begin{remark}\label{rem:F-ou-f}
In the expression \eqref{eq:def_lambdaiN_intro} of the intensity $\lambda_{N, i}$, $X_{N, i}$ given in \eqref{eq:def_UiN} accounts for the stochastic influence of the other interacting neurons, whereas $\eta_t$ represents the deterministic part of the intensity $\lambda_{N, i}$. Having in mind the generic example given in \eqref{eq:def_lambda_generic}, a typical choice would correspond to taking $d=2$ with $\eta:=(\mu, v)$ and 
\begin{equation}
\label{eq:gen_F}
F(X, \eta)= F(X, \mu, v)= \mu + f(v + X)
\end{equation} 
Once again, $\mu$ here corresponds to the spontaneous Poisson activity of the neuron and one may see $v$ as a deterministic part in the evolution of the membrane potential of neuron $i$. Note that we generalize here slightly the framework considered in \cite{CHEVALLIER20191} in the sense that \cite{CHEVALLIER20191} considered \eqref{eq:gen_F} for $\mu\equiv 0$ and $v_t(x)= e^{-\alpha t} v_0(x)$ for some initial membrane potential $v_0(x)$. In the case of \eqref{eq:gen_F}, one retrieves the expression of the macroscopic intensity $\lambda_t(x)$ given in \eqref{eq:def_lambda_lim_generic}. Typical choices of $f$ in \eqref{eq:gen_F} are $f(x)=x$ (the so-called linear model) or some sigmoïd function. Note that there will be an intrinsic mathematical difficulty in dealing with the linear case in this paper, as $f$ is not bounded in this case. As already mentioned in the introduction, for the choice of $h(t)= e^{-\alpha t}$ and $v_t(x)=e^{-\alpha t}v_0(x)$, a straightforward calculation shows that $u_t(x):= v_t(x)+ X_t(x)$ solves  the scalar neural field equation \eqref{eq:NFE} 
with source term $I_t(x)= \int_I W(x,y)\mu(t,y)dy$.

We choose here to work with the generic expression \eqref{eq:def_lambdaiN_intro} instead of \eqref{eq:def_lambda_generic} not only for conciseness of notation, but also to emphasize that the result does not intrinsically depend on the specific form of the function $F$.
\end{remark}

\subsection*{Acknowledgements}

This is a part of my PhD thesis. I would like to thank my PhD supervisors Eric \textsc{Lu\c con} and Ellen \textsc{Saada} for introducing this subject, for their useful advices and for their encouragement. This research has been conducted within the FP2M federation (CNRS FR 2036), and is supported by ANR-19-CE40-0024 (CHAllenges in MAthematical NEuroscience) and ANR–19–CE40–0023 (Project PERISTOCH).

\section{Hypotheses and main results}

\subsection{Hypotheses}

\begin{hyp}\label{hyp_globales}
We assume that
\begin{itemize}
\item $F$ is Lipschitz continuous : there exists $\Vert F \Vert_{L}$ such that for any $x, x'\in  \mathbb{R}$, $\eta,\eta'\in  \mathbb{R}^d$, we have $\vert F(x,\eta) - F(x',\eta') \vert \leq \Vert F \Vert_{L} \left( \vert x-x'\vert +  \vert \eta-\eta'\vert \right)$.
\item $F$ is non decreasing in the first variable, that is for any $\eta\in \mathbb{R}^d$, for any $x, x'\in  \mathbb{R}$ such that $x\leq x'$, one has $F(x,\eta)\leq F(x',\eta)$. Moreover, we assume that $F$ is $\mathcal{C}^2$ on $\mathbb{R}^{d+1}$ with bounded derivatives. We denote by $\partial_x F$ and $\partial_x^2 F$ the partial derivatives of $F$ w.r.t. $x$ and (with some slight abuse of notation) $\partial_\eta F= \left(\partial_{\eta_k}F\right)_{k=1, \ldots d}$ as the gradient of $F$ w.r.t. the variable $\eta\in \mathbb{R}^d$ as well as $\partial_{x, \eta}^2 F= \left(\partial_{x, \eta_k}^2 F\right)_{k=1, \ldots d}$ and $\partial_\eta^2 F= \left(\partial^2_{\eta_k, \eta_l}F\right)_{k,l=1, \ldots d}$ the Hessian of $F$ w.r.t. the variable $\eta$.
\item $\left(\eta_t(x)\right)_{t\geq 0,x\in I}$ is uniformly bounded in $(t,x)$. We also assume that there exists $\eta_\infty$ Lipschitz continuous on $I$ such that 
\begin{equation}\label{eq:def_delta_s}
\delta_t:=\sup_{x\in I} \left| \eta_t(x)-\eta_\infty(x)\right| \xrightarrow[t\to\infty]{}0.
\end{equation}  
\item The memory kernel $h$ is nonnegative and integrable on $[0,+\infty)$.
\item We assume that $W:I^2\to [0,1]$ is continuous.  We refer nonetheless to Section \ref{S:extension} where we show that the results of the paper remain true under weaker hypotheses on $W$.
\end{itemize}
\end{hyp}

It has been showed in \cite{agathenerine2021multivariate} that the process defined in \eqref{eq:def_ZiN} is well-posed, and that the large population limit intensity \eqref{eq:def_lambdabarre}
is well defined in the following sense.
\begin{prop}
\label{prop:exis_H_N} Under Hypothesis \ref{hyp_globales}, for a fixed realisation of the family $\left(\pi_i\right)_{1\leq i \leq N}$,  there exists a pathwise unique multivariate Hawkes process (in the sense of Definition \ref{def:H2}) such that for any $T<\infty$, $\sup_{t\in [0,T]} \sup_{1\leq i \leq N} \mathbf{E}[Z_{N,i}(t)] <\infty$.
\end{prop}
\begin{prop}
\label{prop:exis_lambda_barre}
Let $T>0$. Under Hypothesis \ref{hyp_globales}, there exists a unique solution  $\lambda$ in $\mathcal{C}_b([0,T]\times I, \mathbb{R})$	to \eqref{eq:def_lambdabarre}  and this solution is nonnegative. 
\end{prop}

Both Propositions \ref{prop:exis_H_N} and \ref{prop:exis_lambda_barre} can be found in \cite{agathenerine2021multivariate} as Propositions 2.5 and 2.7 respectively, where $F$ is chosen as $\eta=(\mu, v)$ and $F(x,\eta)=f(x+v)$ with $f$ a Lipschitz function. The same proofs work for our general case $F$. Proposition \ref{prop:exis_lambda_barre} also implies that the limiting spatial profile $X_t$ solving \eqref{eq:def_utx} is well defined.\\

Before writing our next hypothesis, we need to introduce the following integral operator.
\begin{prop}\label{prop:proprietes_TW} Under Hypothesis \ref{hyp_globales}, the integral operator
\begin{equation*}
  \begin{array}{rrcl}
T_W:&   H & \longrightarrow & H \\
   & g & \longmapsto & \left(T_Wg : x \longmapsto \int_I W(x,y) g(y) dy \right)
  \end{array}
\end{equation*}
is continuous in both cases $H=L^\infty(I)$ and $H=L^2(I)$. When $H=L^2(I)$, $T_W$ is compact, its spectrum is the union of $\{0\}$ and a discrete sequence of eigenvalues $(\mu_{n})_{ n\geq1}$ such that $ \mu_{n}\to0$ as $n\to\infty$. Denote by $r_{\infty}=r_\infty(T_W)$, respectively $r_2=r_2(T_W)$ the spectral radii of $T_W$ in $L^\infty(I)$ and $L^2(I)$ respectively. Moreover, we have that
\begin{equation}
\label{eq:spectral_radii_equal}
r_{ 2}(T_{ W})= r_{ \infty}(T_W).
\end{equation}
\end{prop}
The proof can be found in Section \ref{S:proof_det_ut}.

\begin{hyp}\label{hyp:subcritical}
In the whole article, we are in the subcritical case defined by
\begin{equation}\label{eq:def_subcritical}
\left\Vert \partial_x F \right\Vert_\infty \Vert h \Vert _1 r_\infty < 1.
\end{equation}
\end{hyp}

Note that in the complete mean-field case, $W\equiv 1$ and $r_\infty=1$ so that one retrieves the usual subcritical condition as in \cite{delattre2016}. In the linear case $\eta=\mu$ and $F(x, \eta)= \mu +x$, \eqref{eq:def_subcritical} is exactly the subcritical condition stated in \cite{agathenerine2021multivariate}.

The aim of the paper is twofold: firstly, we state a general convergence result as $t\to\infty$ of $X_t$ defined in \eqref{eq:def_utx} (or equivalently $ \lambda_t$ in \eqref{eq:def_lambdabarre}), see Theorem~\ref{thm:large_time_cvg_u_t}. This result is valid for any general kernel $h$ satisfying Hypothesis~\ref{hyp_globales}. Secondly, we address the long-term stability of the microscopic profile $X_N$ defined in \eqref{eq:def_UN}, see Theorem~\ref{thm:long_time}. Contrary to the first one, this second result is stated for the particular choice of the exponential kernel $h$ defined as
\begin{equation}\label{eq:def_exponential}
h(t)=e^{-\alpha t}, \text{with }\alpha>0.
\end{equation}
The parameter $\alpha>0$ is often addressed as the leakage rate.The main advantage of this choice is that the process $X_N$ then becomes Markovian (see e.g. \cite[Section~5]{Ditlevsen2017}). This will turn out to be particularly helpful for the proof of Theorem~\ref{thm:long_time}. As already mentioned in the introduction, \eqref{eq:def_exponential} is the natural framework where to observe the NFE \eqref{eq:NFE} as a macroscopic limit, recall Remark~\ref{rem:F-ou-f}. Note that in the exponential case \eqref{eq:def_exponential}, the subcritical case \eqref{eq:def_subcritical} reads
\begin{equation}\label{eq:def_sub_exp}
\left\Vert \partial_x F \right\Vert_\infty r_\infty < \alpha.
\end{equation}
For our second result (Theorem~\ref{thm:long_time}), we also need some hypotheses on the dilution of the graph. Recall the definition of $\rho_N$ in Definition~\ref{def:espace_proba_bb}.
\begin{hyp}\label{hyp:scenarios} The dilution parameter $\rho_N \in [0, 1]$ satisfies the following dilution condition: there exists $\tau\in (0,\frac{1}{2})$ such that 
\begin{equation}\label{eq:dilution}
N^{1-2\tau}\rho_N^4\xrightarrow[N\to\infty]{}\infty.
\end{equation} 
If one supposes further that $F$ is bounded, we assume the weaker condition
\begin{equation}\label{eq:dilution_Fbounded}
N\rho_N^2\xrightarrow[N\to\infty]{}\infty.
\end{equation}
\end{hyp}

\begin{rem}
Hypothesis~\ref{hyp:scenarios} is stronger than $ \frac{N\rho_N}{\log N}\xrightarrow[N\to\infty]{} \infty$, which is a dilution condition commonly met in the literature concerning LLN results on bounded time intervals for interacting particles on random graphs: it is the same as in \cite{DelattreGL2016,Coppini2019} (and slightly stronger than the optimal $N\rho_N\to +\infty$ obtained in \cite{Coppini_Lucon_Poquet2022} in the case of diffusions and as in \cite{agathenerine2021multivariate} in the case of Hawkes processes).
\end{rem}

\subsection{Main results}

Our first result, Theorem \ref{thm:large_time_cvg_u_t}, studies the limit as $t\to\infty$ of the macroscopic profile $X_t$ (as an element of $\mathcal{C}(I)$) defined in \eqref{eq:def_utx}. Our second result, Theorem \ref{thm:long_time}, focuses on the large time behaviour of $X_N(t)$ defined in \eqref{eq:def_UN} on any time interval of polynomial length. 

\subsubsection{Asymptotic behavior of $(X_t)$}
Recall the definition of $X_t$ in \eqref{eq:def_utx}.
\begin{thm}\label{thm:large_time_cvg_u_t}
Under Hypotheses \ref{hyp_globales} and \ref{hyp:subcritical}, 
\begin{enumerate}[label=(\roman*)]
\item  there exists a unique continuous function $X_\infty:I\mapsto \mathbb{R}^+$ solution of
\begin{equation}\label{eq:def_u_infty}
X_\infty=\Vert h \Vert_1 T_W F\left( X_\infty,\eta_\infty\right).
\end{equation}
\item $\left(X_t\right)_{t\geq 0}$ converges uniformly on $I$ when $t\to\infty$ towards $X_\infty$.
\end{enumerate}
\end{thm}
\begin{remark}\label{rem:correspondance_Xt_ell}
Translating the result of Theorem~\ref{thm:large_time_cvg_u_t} in terms of the macroscopic intensity $\lambda_t$ defined in \eqref{eq:def_lambdabarre} gives immediately that $ \lambda_t$ converges uniformly to $ \ell$ solution to
\begin{equation}
\label{eq:def_l_lim}
    \ell= F \left(\Vert h\Vert_1 T_W \ell, \eta_\infty\right)
\end{equation}
The correspondence between $X_\infty$ and $\ell$ (recall \eqref{eq:def_utx}) is simply given by $X_\infty= \Vert h \Vert_1 T_W \ell$.
\end{remark}

\begin{remark}\label{rem:sys_lin}
In the particular case of an exponential memory kernel \eqref{eq:def_sub_exp}, as a straightforward consequence of the expression of $X_t$ in \eqref{eq:def_utx} and $X_\infty$ in \eqref{eq:def_u_infty}, we have the following differential equation
\begin{equation}\label{eq:dynamic_ut_uinfty}
\partial_t \left( X_t-X_\infty\right)=-\alpha\left( X_t-X_\infty\right) + T_W \left( F(X_t,\eta_t) - F(X_\infty,\eta_\infty)\right).
\end{equation}
A simple Taylor expansion of $X_t$ around $X_\infty$ shows that the linearised system associated to the nonlinear \eqref{eq:dynamic_ut_uinfty} is then
\begin{equation}\label{eq:sys_lin_Y_t}
\partial_t Y_t = -\alpha Y_t + T_W \left( G Y_t \right),
\end{equation}
where 
\begin{equation}\label{eq:def_G}
G:=\partial_x F(X_\infty,\eta_\infty).
\end{equation}
\end{remark}
The subcritical condition \eqref{eq:def_sub_exp} translates into the existence of a spectral gap for the linear dynamics \eqref{eq:sys_lin_Y_t}, which makes the stationary point $X_\infty$ linearly stable. More precisely,

\begin{prop}\label{prop:operateur_L} Assume that the memory kernel $h$ is exponential \eqref{eq:def_sub_exp}. Define the linear operator 
\begin{equation} \label{eq:def_operator_L}
  \begin{array}{rrcl}
\mathcal{L}:&    L^2 (I) & \longrightarrow &L^2 (I) \\
   & g & \longmapsto & \mathcal{L}(g)=-\alpha g + T_W( Gg).
  \end{array}
\end{equation}
Then under Hypotheses \ref{hyp_globales} and \ref{hyp:subcritical}, $\mathcal{L}$ generates a contraction semi-group on $L^2(I)$ $\left(e^{t\mathcal{L}}\right)_{t\geq 0}$ such that for any $g\in L^2(I)$
\begin{equation}\label{eq:contraction_sg}
\Vert e^{t\mathcal{L}} g \Vert_2 \leq e^{-t\gamma} \Vert g \Vert_2,
\end{equation}
where
 \begin{equation}\label{eq:def_gamma}
\gamma:=\alpha-r_\infty \left\Vert \partial_uF\right\Vert_\infty>0.
\end{equation}
\end{prop}

\subsubsection{Long-term stability of the microscopic spatial profile}

From now on, we place ourselves in the exponential case \eqref{eq:def_sub_exp}. We first state a convergence result of $X_N$ towards the macroscopic $X$ on a bounded time interval $[0, T]$.
\begin{prop}\label{prop:finite_time}
Let $T>0$.  Under Hypotheses \ref{hyp_globales}, \ref{hyp:subcritical} and   \ref{hyp:scenarios}, for any $\varepsilon>0$,  $ \mathbb{ P}$-a.s.
\begin{equation}\label{eq:finite_time}
\mathbf{P}\left( \sup_{t\in [0,T]}\left\Vert X_N(t)-X_t \right\Vert_2\geq \varepsilon\right) \xrightarrow[N\to\infty]{}0.
\end{equation}
\end{prop}
Note that Proposition~\ref{prop:finite_time} slightly generalises \cite[Prop. 3.17]{agathenerine2021multivariate} (see also \cite[Cor.~2]{CHEVALLIER20191} for a similar result) where it is proven that $\mathbf{E}\left[ \int_0^T\int_I \left\vert X_N(t)(x)-X_t(x)\right\vert dx~ dt\right]\xrightarrow[N\to\infty]{}0$ for any $T>0$. Here, we are more precise as we show uniform convergence of $X_N(t)$ in $L^2(I)$ instead of $L^1(I)$.

We are now in position to state the main result of the paper: the proximity stated in Proposition~\ref{prop:finite_time} is not only valid on a bounded time interval, but propagates to arbitrary polynomial times in $N \rho_N$.

\begin{thm}\label{thm:long_time} 
Choose some $t_{f}>0$ and $m\geq 1$. Then, under Hypotheses \ref{hyp_globales}, \ref{hyp:scenarios} and  \ref{hyp:subcritical}, $ \mathbb{ P}$-a.s. for $ \varepsilon>0$ small enough, 
\begin{equation}\label{eq:long_time_pol}
\mathbf{ P} \left( \sup_{ t\in \left[ t_{ \varepsilon}, (N \rho_{ N})^{ m} t_{ f}\right]} \left\Vert X_{ N}(t) - X_{ \infty}\right\Vert_2\geq \varepsilon\right) \xrightarrow[ N\to\infty]{}0.
\end{equation}
for some $ t_{\varepsilon}>0$ independent on $N$.
\end{thm}

Since $F$ is Lipschitz and $\lambda_{N,i}(t)= F(X_{N,i}(t-), \eta_t(x_i))$ by \eqref{eq:def_lambdaiN_intro}, it is straightforward to derive from Theorem~\ref{thm:long_time} a similar result for the profile of densities
\begin{equation}\label{eq:def_lambdaN}
\lambda_{N}(t)(x):=\sum_{i=1}^N \lambda_{N,i}(t) \mathbf{1}_{x\in\left(\frac{i-1}{N}, \frac{i}{N}\right]},\ x\in I.
\end{equation}

\begin{cor}\label{cor:long_time_lambda}
Recall the definition of $\ell$ in \eqref{eq:def_l_lim}. Under the same set of hypotheses of Theorem \ref{thm:long_time} and with the same notation,
\begin{equation}\label{eq:long_time_pol_lambda}
\mathbf{ P} \left( \sup_{ t\in \left[ t_{ \varepsilon}, (N \rho_{ N})^{ m} t_{ f}\right]} \left\Vert \lambda_{N}(t) - \ell\right\Vert_2\geq \varepsilon\right) \xrightarrow[ N\to\infty]{}0.
\end{equation}
\end{cor}

\subsection{Examples and extensions}
We give here some illustrating examples of our main results.

\subsubsection{Mean-field framework} To the best of the knowledge of the author, already in the simple homogeneous case of mean-field interaction, there exists no long-term stability result such as Theorem~\ref{thm:long_time}. We stress that our result may have an interest of its own in this case. Let us be more specific. When $\rho_N=W_N=1$ and $\mu_t(x)=\mu\geq 0$, the process introduced in Definition \ref{def:H2} reduces to the usual mean-field framework \cite{delattre2016}:
\begin{equation}\label{eq:def_ZiN_CM}
Z_{N,i}(t) = \int_0^t \int_0^\infty \mathbf{1}_{\{z\leq \lambda_{N}(s)\}} \pi_i(ds,dz)
\end{equation}
with $\lambda_{N}(t)$ defined by
\begin{equation}\label{eq:def_lambdaiN_intro_CM}
\lambda_{N}(t)= F(X_{N}(t-), \eta),
\end{equation}
where
\begin{equation}\label{eq:def_UiN_CM}
X_{N}(t)=\sum_{j=1}^N \dfrac{1}{N}\int_0^{t} h(t-s) dZ_{N,j}(s),
\end{equation}
In this simple case, the spatial framework is no longer useful (in particular the spatial profile defined in \eqref{eq:def_UN} is constant in $x$ so that the $L^2$ framework is not relevant, one has only to work in $\mathbb{R}$). The macroscopic intensity and synaptic current (respectively \eqref{eq:def_lambdabarre} and \eqref{eq:def_utx} become
\begin{equation}\label{eq:def_macro_CM}
X_t:=\int_0^t h(t-s)\lambda_sds,\quad \lambda_t:=F(X_t,\eta).
\end{equation}
The main results of the paper translate then into
\begin{thm}\label{thm:CM}
Under Hypothesis \ref{hyp_globales} and when $\left\Vert \partial_x F \right\Vert_\infty \Vert h \Vert _1 < 1$, 
there exists a unique $X_\infty\in \mathbb{R}_+$ solution to
$X_\infty=\Vert h \Vert_1 F\left(X_\infty,\eta\right)$, and
 $\left(X_t\right)_{t\geq 0}$ converges when $t\to\infty$ towards $X_\infty$. Respectively, $(\lambda_t)_{t\geq 0}$ converges towards $\ell$, the unique solution to $\ell=F\left(\Vert h \Vert_1 \ell,\eta\right)$. 
Moreover, under the same hypotheses, in the exponential case \eqref{eq:def_exponential},
for any $t_{f}>0$ and $m\geq 1$, $ \mathbb{ P}$-a.s. for $ \varepsilon>0$ small enough, 
$\displaystyle \mathbf{ P} \left( \sup_{ t\in \left[ t_{ \varepsilon}, N ^{ m} t_{ f}\right]} \left\vert X_{ N}(t) - X_{ \infty}\right\vert\geq \varepsilon\right)$ and $\mathbf{ P} \left( \sup_{ t\in \left[ t_{ \varepsilon}, N ^{ m} t_{ f}\right]} \left\vert \lambda_{ N}(t) - \ell\right\vert\geq \varepsilon\right)$ tend to $0$ as $N\to\infty$ for some $ t_{\varepsilon}>0$ independent on $N$. 
\end{thm}

\begin{remark}\label{rem:result_CM}
The previous result applies in particular to the linear case where $\eta= \mu$ and $F(x,\eta)=\mu+x$. We have then that $\ell=\dfrac{\mu}{1-\Vert h \Vert_1}$ in this case, as in \cite{delattre2016}. 
\end{remark}

\subsubsection{Erd\H{o}s-R\'enyi graphs} An immediate extension of the last mean-field case concerns the case of homogeneous Erd\H{o}s-R\'enyi graphs: choose $W_N(x,y)=\rho_N$ for all $x,y\in I$. The results of our paper are valid under the dilution Hypothesis \ref{hyp:scenarios}. It is however likely that these dilution conditions are not optimal (compare with the result of \cite{Coppini2022} with the condition $N\rho_N\to\infty$ in the diffusion case, but a difficulty here is that we deal with a multiplicative noise whereas it is essentially additive in \cite{Coppini2022}).

\subsubsection{Examples in the inhomogeneous case} As already mentionned in Hypothesis \ref{hyp_globales}, the results are valid for any $W$ continuous, interesting examples include $W(x,y)=1-\max(x,y)$, $W(x,y)= 1-xy$, see \cite{borgs2011,Borgs2018}. Note also that we do not suppose any symmetry on $W$. Another rich class of examples concerns the \emph{Expected Degree Distribution} model \cite{Chung2002,Ouadah2019} where $W(x,y)=f(x)g(y)$ for any continuous functions $f$ and $g$ on $I$. The specificity of such class is that we have an explicit formulation of $r_\infty$, that is $r_\infty= \int_I f(x)g(x)dx$ when $\int_I g =1$. In the linear case, we obtain an explicit formula for $\lambda_t$ in \cite[Example 4.3]{agathenerine2021multivariate}.

\subsubsection{Extensions}\label{S:extension}

It is apparent from the proofs below that one can weaken the hypothesis of continuity of $W$. Under the hypothesis that $W$ is bounded, Proposition \ref{prop:exis_lambda_barre} remains true when $\mathcal{C}_b([0,T]\times I)$ is replaced by $\mathcal{C}\left( [0,T],  L^\infty(I)\right)$ (continuity of $\lambda_t$ and $X_t$ in $x$ may not be satisfied). Supposing further that there exists a partition of $I$ into $p$ intervals $I=\sqcup_{k=1,\cdots,p} C_k$ such that for all $\epsilon>0$, there exists $\eta>0$ such that $\int_I \left\vert W(x,y)-W(x',y)\right\vert dy <\epsilon$ when $\vert x-x'\vert<\eta$ and $x,x'\in C_k$, then for every $k$, $\lambda_{\vert [0,T]\times\mathring{C_k}} $ and $X_{\vert [0,T]\times\mathring{C_k}}$ are both continuous. When $p=1$, both $\lambda$ and $X$ are continuous on $[0,T]\times I$.

Concerning Theorem \ref{thm:large_time_cvg_u_t}, defining for $k\in \{1,2\}$:
\begin{equation}\label{eq:def_Rnk}
R^W_{N,k}:=\dfrac{1}{N} \sum_{i,j=1}^N \int_{B_{N,j}} \left\vert W(x_i,x_j)-W(x_i,y)\right\vert^k dy,
\end{equation}
and
\begin{equation}\label{eq:def_Sn}
S^W_{N}:=\sum_{i=1}^N \int_{B_{N,i}} \left(\int_I \left\vert W(x_i,y) - W(x,y)\right\vert^2 dy\right)dx,
\end{equation}
Theorem \ref{thm:long_time} remains true when $R^W_{N,1},R^W_{N,2}, S^W_{N} \xrightarrow[N\to\infty]{}0$, see Lemmas \ref{lem:drift_term_phi_N1}, \ref{lem:drift_term_phi_N2} and \ref{lem:drift_term_phi_N3}.

These particular conditions are met in the following cases (details of the computation are left to the reader)
\begin{itemize}
\item  P-nearest neighbor model \cite{Omelchenko2012}: $W(x,y)=\mathbf{1}_{d_{\mathcal{S}_1}(x,y)< r}$ for any $(x,y)\in I^2$ for some fixed $r\in (0,\frac{1}{2})$, with
$d_{\mathcal{S}_1}(x,y)=\min(\vert x-y \vert,1-\vert x-y \vert)$.
\item Stochastic block model \cite{holland1983,Ditlevsen2017}: it corresponds to considering $p$ communities $(C_k)_{1\leq k \leq p}$. An element of the community $C_l$ communicates with an element of the community $C_k$ with probability $p_{kl}$. This corresponds to the choice of interaction kernel $W(x,y)=\sum_{k,l}p_{kl}\mathbf{1}_{x\in C_k, y\in C_l}$.
\end{itemize}


\subsection{Link with the literature}\label{S:litterature}

Several previous works have complemented the propagation of chaos result mentioned in \eqref{eq:chaos_generic} in various situations: Central Limit Theorems (CLT) have been obtained in \cite{delattre2016,Ditlevsen2017} for homogeneous mean-field Hawkes processes (when both time and $N$ go to infinity) or with age-dependence in \cite{Chevallier2017}. One should also mention the functional fluctuation result recently obtained in \cite{Heesen2021}, also in a pure mean-field setting. A result closer to our case with spatial extension is \cite{ChevallierOst2020}, where a functional CLT is obtained for the spatial profile $X_{ N}$ around its limit. Some insights of the necessity of considering stochastic versions of the NFE \eqref{eq:NFE} as second order approximations of the spatial profile are in particular given in \cite{ChevallierOst2020}. Note here that all of these works provide approximation results of quantities such that $ \lambda_{ N}$ or $X_{ N}$ that are either valid on a bounded time interval $[0, T]$ or under strict growth condition on $T$ (see in particular the condition $ \frac{ T}{ N} \to 0$ for the CLT in \cite{Ditlevsen2017}), whereas we are here concerned with time-scales that grow polynomially with $N$.
  
The analysis of mean-field interacting processes on long time scales has a significant history in the case of interacting diffusions. The important issue of uniform propagation of chaos has been especially studied mostly in reversible situations (see e.g. the case of granular media equation \cite{Bolley:2013}) but also more recently in some irreversible situations (see \cite{Colombani2022}). There has been in particular a growing interest in the long-time analysis of phase oscillators (see \cite{giacomin_poquet2015} and references therein for a comprehensive review on the subject). We do not aim here to be exhaustive, but as the techniques used in this work present some formal similarities, let us nonetheless comment on the analysis of the simplest situation, i.e. the Kuramoto model. One is here interested in the longtime behavior of the empirical measure $ \mu_{ N, t}:= \frac{ 1}{ N} \sum_{ i=1}^{ N} \delta_{ \theta_{ i, t}}$ of the system of interacting diffusions $(\theta_{ 1}, \ldots, \theta_{ N})$ solving the system of coupled SDEs $ {\rm d} \theta_{ i,t}= - \frac{ K}{ N} \sum_{ j=1}^{ N} \sin( \theta_{ i,t}- \theta_{ j,t}){\rm d} t + {\rm d}B_{ i, t}$. Standard propagation of chaos techniques show that $ \mu_{ N}$ converges weakly on a bounded time interval $[0, T]$ to the solution $ \mu_{ t}$ to the nonlinear Fokker-Planck (NFP) equation $\partial_t \mu_t\, =\, \frac{1}{2} \partial_{ \theta}^{ 2} \mu_t+K\partial_\theta \Big( \mu_t(\sin * \mu_t)\Big)$. The simplicity of the Kuramoto model lies in the fact that one can easily prove the existence of a phase transition for this model: when $K\leq 1$, $ \mu\equiv \frac{ 1}{ 2\pi}$ is the only (stable) stationary point of the previous NFP (subcritical case), whereas it coexists with a stable circle of synchronised profiles when $K>1$ (supercritical case). A series of papers have analysed the longtime behavior of the empirical measure $\mu_N$ of the Kuramoto model (and extensions) in both the subcritical and supercritical cases (see in particular \cite{bertini14,lucon_poquet2017,giacomin2012,Coppini2022} and references therein). The main arguments of the mentioned papers lie in a careful analysis of two contradictory phenomena that arise on a long-time scale: the stability of the deterministic dynamics around stationary points (that forces $ \mu_{ N}$ to remain in a small neighborhood of these points) and the presence of noise in the microscopic system (which makes $ \mu_{ N}$ diffuse around these points). In particular, the work that is somehow formally closest to the present article is \cite{Coppini2022}, where the long-time stability of $ \mu_{ N}$ is analysed in both sub and supercritical cases for Kuramoto oscillators interacting on an Erd\H{o}s-R\'enyi graph. We are here (at least formally) in a similar situation to the subcritical case of \cite{Coppini2022}: the deterministic dynamics of the spatial profile $X_{ N}$ (given by \eqref{eq:def_UN}) has a unique stationary point which possesses sufficient stability properties. The point of the analysis relies then on a time discretization and some careful control on the diffusive influence of noise that competes with the deterministic dynamics. The main difference (and present difficulty in the analysis) with the diffusion case in \cite{Coppini2022} is that our noise (Poissonnian rather than Brownian) is multiplicative (whereas it is essentially additive in \cite{Coppini2022}). This explains in particular the stronger dilution conditions that we require in Hypothesis~\ref{hyp:scenarios} (compared to the optimal $N \rho_{ N}\to \infty$ in \cite{Coppini2022}) and also the fact that we only reach polynomial time scales (compared to the sub-exponential scale in \cite{Coppini2022}). There is however every reason to believe that the stability result of Theorem~\ref{thm:long_time} would remain valid up to this sub-exponential time scale.

Note here that we deal directly with the control of the Poisson noise. Another possibility would have been to use some Brownian approximation of the dynamics of $X_{ N}$. Some results in this direction have been initiated in \cite{Ditlevsen2017} for spatially-extended Hawkes processes exhibiting oscillatory behaviors: some diffusive approximation of the dynamics of the (equivalent of) the spatial profile is provided (see \cite[Section~5]{Ditlevsen2017}). Note however that this approximation is based on the comparison of the corresponding semigroups and is not uniform in time. Hence, it is unclear how one could exploit these techniques for our case. Some stronger (pathwise) approximations between Hawkes and Brownian dynamics have been further proposed in \cite{chevallierMT2021}, based on Koml\'os, Major and Tusn\'ady (KMT) coupling techniques (\cite{ethier_kurtz1986}, see also \cite{Prodhomme_arxiv2020} for similar techniques applied to finite dimensional Markov chains). However, this approximation is again not uniform in time so that applying this coupling to our present case is unclear. Our proof is more direct and does not rely on such Brownian coupling. To the author’s knowledge, this is the first result on large time stability of Hawkes process (not mentioning the issue of the random graph of interaction, we believe that our results remain also relevant in the pure mean-field case, see Theorem~\ref{thm:CM}).


\subsection{Strategy of proof and organization of the paper}

Section~\ref{S:proof_det_ut} is devoted to prove the convergence result as $t\to\infty$ of Theorem \ref{thm:large_time_cvg_u_t}. This in particular requires some spectral estimates on the operator $\mathcal{L}$ defined in Proposition~\ref{prop:operateur_L} that are gathered in Section~\ref{S:proofs_operator_L}.

The main lines of proof for Theorem \ref{thm:long_time} are given in Section~\ref{S:proof_largetime}.The strategy of proof is sketched here:
\begin{enumerate}
\item The starting point of the analysis is a semimartingale decomposition of $Y_N:= X_N- X$, detailed in Section~\ref{S:mild_formulation}. The point is to decompose the dynamics of $Y_N$ in terms of, at first order, the linear dynamics \eqref{eq:sys_lin_Y_t} governing the behavior of the deterministic profile $X$, modulo some drift terms coming from the graph and its mean-field approximation, some noise term and finally some quadratic remaining error coming from the nonlinearity of $F$.
\item A careful control on each of these terms in the semimartingale expansion on a bounded time interval are given in the remaining of Section~\ref{S:mild_formulation}. The proof of these estimates are respectively given in Section~\ref{S:proof_NP} (for the noise term) and Section~\ref{S:proof_drift} (for the drift term).
\item The rest of Section~\ref{S:proof_largetime} is devoted to the proof of Theorem~\ref{thm:long_time}, see Section~\ref{S:proof_mainTh}. The first point is that for any given $\varepsilon>0$, one has to wait a deterministic time $t_\varepsilon>0$, so that the deterministic profile $X_t$ reaches an $\varepsilon$-neighborhood of $X_\infty$. It is easy to see from the spectral gap estimate \eqref{eq:contraction_sg} that this $t_\varepsilon$ is actually of order $\frac{-\log(\varepsilon)}{\gamma}$. Then, using Proposition~\ref{prop:finite_time}, the microscopic process $X_N$ is itself $\varepsilon$-close to $X_\infty$ with high-probability.
\item The previous argument is the starting point of an iterative procedure that works as follows: the point is to see that provided $X_N$ is initially close to $X_\infty$, it will remain close to $X_\infty$ on some $[0, T]$ for some sufficiently large deterministic $T>0$. The key argument is that on a bounded time interval, the deterministic linear dynamics dominates upon the contribution of the noise, so that one has only to wait some sufficiently large $T$ so that the deterministic dynamics prevails upon the other contributions.
\item The rest of the proof consists in an iterative procedure from the previous argument, taking advantage of the Markovian structure of the dynamics of $X_N$. The time horizon at which one can pursue this recursion is controlled by moment estimates on the noise, proven in Section~\ref{S:proof_NP}.
\end{enumerate}
The rest of the paper is organised as follows: Section~\ref{S:finite} collects the proofs for the finite time behavior of Proposition~\ref{prop:finite_time} whereas some technical estimates are gathered in the appendix.


\section{Asymptotic behavior of $(X_t)$}\label{S:proof_det_ut}
This section is related to the proof of Theorem~\ref{thm:large_time_cvg_u_t}.
\subsection{Estimates on the operator $\mathcal{L}$}
\label{S:proofs_operator_L}
\begin{proof}[Proof of Proposition \ref{prop:proprietes_TW}] The continuity and compactness of $T_W$ come from the boundedness of $W$. The structure of the spectrum of $T_W$ is a consequence of the spectral theorem for compact operators. The equality between the spectral radii is postponed to Lemma \ref{lem:op_radius} where a more general result is stated (see also Proposition 4.7 of \cite{agathenerine2021multivariate} for a similar result).
\end{proof}

\begin{proof}[Proof of Proposition  \ref{prop:operateur_L}]
Let us introduce the operator 
\begin{equation} \label{eq:def_operator_U}
  \begin{array}{rrcl}
\mathcal{U}:&    L^2 (I) & \longrightarrow &L^2 (I) \\
   & g & \longmapsto & \mathcal{U}(g)=T_W( Gg),
  \end{array}
\end{equation}
we have then $\mathcal{L}=-\alpha Id + \mathcal{U}$. By Hypothesis \ref{hyp_globales}, $G$ is bounded. Then, for any $g\in L^2(I)$ using Cauchy-Schwarz inequality, $\Vert\mathcal{U}(g)\Vert_2^2\leq \Vert W \Vert_2^2 \Vert G \Vert_\infty \Vert g \Vert_2^2$. The operator $\mathcal{U}$ is then compact and thus has a discrete spectrum. Moreover, $r_2(\mathcal{U})=r_\infty(\mathcal{U})$, see Lemma \ref{lem:op_radius}, and $r_\infty(\mathcal{U}) \leq r_\infty(T_W) \Vert G \Vert_\infty$ as for any $g\in L^\infty$ and $x\in I$, $\vert\mathcal{U} g(x)\vert \leq \Vert T_W \Vert_\infty \Vert Gg \Vert_\infty \leq \Vert T_W \Vert_\infty \Vert G \Vert_\infty \Vert g \Vert_\infty$. Then $\mathcal{L}$ also has a discrete spectrum, which is the same as $\mathcal{U}$ but shifted by $\alpha$. Since $r_2(\mathcal{U})=r_\infty(\mathcal{U})$ (see Lemma \ref{lem:op_radius}), for any $\mu\in \sigma(\mathcal{L})\setminus\{0\}$, $ \vert \mu + \alpha \vert \leq  r_\infty(\mathcal{U}) $ thus $Re(\mu)\leq -\alpha + r_\infty(\mathcal{U})\leq -\alpha + r_\infty \Vert \partial_uF\Vert_\infty<0$ by \eqref{eq:def_subcritical}. The estimate \eqref{eq:contraction_sg} follows then from functional analysis (see e.g. Theorem 3.1 of \cite{Pazy1974}).
\end{proof}

\subsection{About the large time behavior of $X_t$}
\begin{proof}[Proof of Theorem \ref{thm:large_time_cvg_u_t}]
We prove that 
\begin{itemize}
\item there exists a unique function $\ell:I\mapsto \mathbb{R}^+$ solution of \eqref{eq:def_l_lim}, continuous and bounded on $I$, and that
\item $\left(\lambda_t\right)_{t\geq 0}$ converges uniformly when $t\to\infty$ towards $\ell$.
\end{itemize} 
It gives then that $X_\infty:=\Vert	h\Vert_1 T_W\ell$ is the  unique solution of \eqref{eq:def_u_infty}. Then, as $X_t(x)=\int_{I} W(x,y)\int_0^t h(t-s) \lambda_s(y)ds~ dy$, as $\left(\lambda_t\right)$ is uniformly bounded, and as $h$ is integrable and $\lambda_t\to \ell$ uniformly, we conclude by dominated convergence that uniformly on $y$, $\int_0^t h(t-s) \lambda_s(y)ds \xrightarrow[t\to\infty]{} \Vert h \Vert_1 \ell(y)$. As $T_W$ is continuous, the result follows: $X_t$ converges uniformly towards $X_\infty$. We show first that $(\lambda_t)$ is uniformly bounded. Let $\overline{\lambda}_t(x)=\sup_{s\in [0,t]} \lambda_s(x)$, we have then with \eqref{eq:def_lambdabarre}, for $s\in[0, t]$
\begin{align*}
\lambda_s(x)&\leq F(0,0) + \Vert F \Vert_L \vert\eta_s(x)\vert+\Vert \partial_x F \Vert_\infty\int_{I} W(x,y)\int_0^s h(s-u) \lambda_u(y)dudy \\
&\leq F(0,0) + \Vert F \Vert_L \sup_{s,x}\vert\eta_s(x)\vert + \Vert \partial_x F \Vert_\infty \Vert h \Vert_1 T_W \overline{\lambda}_t(x),
\end{align*}
hence $\overline{\lambda}_t(x)\leq C_{F,\eta} + \Vert \partial_x F \Vert_\infty \Vert h \Vert_1 T_W \overline{\lambda}_t(x)$. An immediate iteration gives then $\overline{\lambda}_t(x) \leq C_{F,\eta,n_0,h} + \Vert \partial_x F \Vert_\infty^{n_0} \Vert h \Vert_1^{n_0} \left\vert T_W^{n_0}  \overline{\lambda}_t(x) \right\vert$, so that, by \eqref{eq:def_subcritical} and choosing $n_0$ sufficiently large such that $\Vert \partial_x F \Vert_\infty^{n_0} \Vert h \Vert_1^{n_0} \Vert T_W \Vert^{n_0}<1$, we obtain that $ \Vert \overline{\lambda}_t \Vert_\infty <C$, where $C$ is independent of $t$. Passing to the limit as $t\to\infty$, this implies that $\left(\lambda_t\right)_{t>0}$ is then uniformly bounded, i.e. $\sup_{t\geq 0}\sup_{x\in I} \left\vert \lambda_t(x)\right\vert=:\Vert \lambda\Vert_\infty <\infty$.

We show next that $\left(\lambda_t\right)$ converges pointwise. We start by studying the supremum limit of $\lambda_t$, denoted by $\overline{\ell}(x):= \limsup_{t\to\infty} \lambda_t(x) = \inf_{r>0}\sup_{t>r} \lambda_t(x) =: \inf_{r>0} \Lambda(r,x)$. Then for any $r>0$ and $t>r$:
\begin{align*}
\lambda_t(x) &= F\left( \int_{I} W(x,y)\int_0^r h(t-s) \lambda_s(y)ds~ dy + \int_{I} W(x,y)\int_r^t h(t-s) \lambda(s,y)ds~ dy, \eta_t(x) \right)
\\
&\leq F\left( \int_{I} W(x,y)\int_0^r h(t-s) \lambda_s(y)ds~ dy + \int_{I} W(x,y)\Lambda(r,y)\int_r^t h(t-s) ds~ dy, \eta_t(x) \right)
\end{align*}
by monotonicity of $F$ in the first variable and by positivity of $W$ and $h$. As $\int_r^t h(t-s)ds\leq \Vert h \Vert_1$, it gives
$$\lambda_t(x) \leq F\left( \int_{I} W(x,y)\int_0^r h(t-s) \lambda_s(y)ds~ dy +  \Vert h \Vert_1 \int_{I} W(x,y) \Lambda(r,y)dy, \eta_t(x) \right),$$
and as $h(t)\to 0$, by dominated convergence $\int_{I} W(x,y)\int_0^r h(t-s) \lambda_s(y)ds~ dy \xrightarrow[t\to\infty]{}0$ and by continuity and monotonicity of $F$, we obtain
\begin{equation}
\label{eq:limsup_l}
    \overline{\ell}(x) \leq F\left( \Vert h \Vert_1 \left(T_W\overline{\ell}\right)(x),\eta_\infty(x)\right).
\end{equation}
Note that $\Vert \overline{\ell}\Vert_\infty \leq \Vert \lambda \Vert_\infty<\infty$, by the first step of this proof. Denote in a same way  $\underline{\ell}(x):= \liminf_{t\to\infty} \lambda_t(x) = \sup_{r>0}\inf_{t>r} \lambda_t(x) =: \sup_{r>0} v(r,x)$, for any $t>0$ we have by monotonocity of $F$ in the first variable:
\begin{align*}
\lambda_t(x) &= F\left( \int_0^{\frac{t}{2}}\int_{I} W(x,y)h(t-s) \lambda_s(y) dyds + \int_{\frac{t}{2}}^t\int_{I} W(x,y)h(t-s) \lambda_s(y) dyds, \eta_t(x) \right)\\ 
&\geq  F\left( \int_{\frac{t}{2}}^t\int_{I} W(x,y)h(t-s) v\left(\frac{t}{2}, y\right) dyds, \eta_t(x) \right)\\
&= F\left(\int_{0}^{\frac{t}{2}}h(u)du \int_{I} W(x,y)v\left(\frac{t}{2}, y\right) dy, \eta_t(x) \right).
\end{align*}
Taking $\liminf_{t\to\infty}$ on both sides, by monotone convergence, we obtain 
\begin{equation}
\label{eq:liminf_l}
\underline{\ell}(x) \geq F\left( \Vert h \Vert_1 \left(T_W \underline{\ell}\right)(x), \eta_\infty(x) \right).
\end{equation}
Combining \eqref{eq:limsup_l} and \eqref{eq:liminf_l}, setting $H: l \in L^\infty \mapsto F\left( \Vert h \Vert_1 T_W l, \eta_\infty\right)\in L^{\infty}$, we have shown 
\begin{equation}
\label{eq:control_Hl}
    H \underline{\ell} \leq \underline{\ell} \leq \overline{\ell} \leq H\overline{\ell}.
\end{equation}
For any $l$ and $l'$ in $L^\infty(I)$ and any $x\in I$, we have
\begin{align*}
\vert Hl(x) - Hl'(x) \vert &\leq \left| F\left( \Vert h \Vert_1 \left(T_W l\right)(x), \eta_\infty(x)\right) - F\left( \Vert h \Vert_1 \left(T_W l'\right)(x), \eta_\infty(x)\right) \right|\\
&\leq \left\Vert \partial_x F \right\Vert_\infty \Vert h \Vert_1 \left|\left( T_W ( l - l') \right) (x)\right|.
\end{align*} 
By iteration we show that $\Vert H^{n_0}l - H^{n_0}l' \Vert_\infty \leq \left\Vert \partial_u F \right\Vert_\infty^{n_0}\Vert h \Vert _1^{n_0} \Vert T_W^{n_0} \Vert  \Vert l - l' \Vert_\infty$, so that, choosing again $n_0$ sufficiently large, $H^{n_0}$ is a contraction mapping by \eqref{eq:def_subcritical}. Hence, by \eqref{eq:control_Hl}, one has necessarily that for all $x\in I$ $\underline{\ell}(x)=\overline{\ell}(x)<+\infty$ thus $(\lambda_t)$ converges pointwise towards $\ell=\underline{\ell}=\overline{\ell}$ the unique fixed point of $H$ which satisfies \eqref{eq:def_l_lim}.

We show now that the family $\left(\lambda_t\right)_{t\geq 0}$ is equicontinuous so that the pointwise convergence will imply uniform convergence on the compact set $I$. For any $(x,y)\in I$ and $t\geq 0$, we have
\begin{align*}
\left\vert \lambda_t(x) - \lambda_t(y)\right\vert&=\left\vert F(X_t(x),\eta_t(x))-F(X_t(y),\eta_t(y)\right\vert\\
&\leq \Vert F \Vert_L \left( \left\vert X_t(x)-X_t(y)\right\vert + \left\vert \eta_t(x)-\eta_t(y)\right\vert\right).
\end{align*}
With \eqref{eq:def_delta_s}, we have
\begin{align*}
\left\vert \eta_t(x)-\eta_t(y)\right\vert&\leq \left\vert \eta_t(x)-\eta_\infty(x)\right\vert+\left\vert \eta_\infty(x)-\eta_\infty(y)\right\vert+\left\vert \eta_\infty(y)-\eta_t(y)\right\vert\\
&\leq 2\delta_t + \Vert \eta_\infty\Vert_L \vert x-y\vert,
\end{align*}
and as $\lambda$ is bounded, we have
\begin{align}\label{eq:reg_Xt_aux}
\left\vert X_t(x)-X_t(y)\right\vert  &= \left\vert \int_I \left(W(x,z)-W(y,z)\right)\int_0^t h(t-s)\lambda_s(z)dsdz\right\vert\notag\\
&\leq \Vert \lambda \Vert_\infty \Vert h \Vert_1 \int_I \left\vert W(x,z)-W(y,z)\right\vert dz.
\end{align} 
Then $\vert \lambda_t(x)-\lambda_t(y)\vert\leq C_{F,\lambda,h,W}\left(\delta_t+\vert x-y\vert+\int_I \left\vert W(x,z)-W(y,z)\right\vert dz\right)$. Fix $\varepsilon >0$, with \eqref{eq:def_delta_s}, one can find $T$ such that $ C_{F,\lambda,h,W}\delta_t\leq \dfrac{\varepsilon}{2}$ for any $t\geq T$, and as $W$ is uniformly continuous on $I^2$, one can find $\eta>0$ such that $C_{F,\lambda,h,W}\left(\vert x-y\vert+\int_I \left\vert W(x,z)-W(y,z)\right\vert dz\right)\leq \dfrac{\varepsilon}{2}$ when $\vert x-y\vert\leq \eta$. We can divide $[0,1]$ in intervals $[z_k,z_{k+1}]$ such that for any $k$, $ z_{k+1}-z_k \leq \eta$. Then, for any $x\in [0,1]$, one can find $z_k$ such that $\vert z_k-x\vert\leq \eta$, and $\vert \lambda_t(x)-\ell(x)\vert\leq \vert  \lambda_t(x)-\lambda_t(z_k)\vert + \vert \lambda_t(z_k)-\ell(z_k)\vert+\vert\ell(z_k)-\ell(x)\vert$. By pointwise convergence, $\vert\lambda_t(z_k)-\ell(z_k)\vert\leq\varepsilon$ for $t$ large enough (but independent of the choice of $x$), and $\vert \ell(z_k)-\ell(x)\vert\leq \varepsilon$ by taking the limit when $t\to\infty$ in $\vert\lambda_t(z_k)-\lambda_t(x)\vert\leq \varepsilon$. 
It gives then $\vert \lambda_t(x)-\ell(x)\vert\leq 3\varepsilon$ hence $\sup_{x\in I}\vert \lambda_t(x)-\ell(x)\vert \xrightarrow[t\to\infty]{}0$, 
i.e. $\left(\lambda_t\right)$ converges uniformly towards $\ell$. Similarly to \eqref{eq:reg_Xt_aux}, for any $x,x'\in I$,
$$\left\vert X_\infty(x) - X_\infty(x') \right\vert \leq \Vert h \Vert_1 \Vert \ell \Vert_\infty \int_I \left\vert   W(x,y) - W(x',y) \right\vert dy$$
which gives, as $W$ is uniformly continous, the continuity of $X_\infty$.
\end{proof}
\section{Large time behavior of $U_N(t)$}\label{S:proof_largetime}
The aim of the present section is to prove Theorem~\ref{thm:long_time}.
To study the behavior of $\left\Vert  X_N(t) - X_\infty\right\Vert_2$, let
\begin{equation}\label{eq:def_Y_N}
Y_N:=X_N-X_\infty.
\end{equation} 
The first step is to write the semimartingale decomposition of $Y_N$, written in a mild form (see Section~\ref{S:mild_formulation}). The proper control on the drift and noise terms are given in Propositions~\ref{prop:noise_perturbation} and~\ref{prop:drift_term}. In Section~\ref{S:proof_mainTh}, we give the proof of Theorem \ref{thm:long_time}, based in particular on the convergence on a bounded time interval in Proposition~\ref{prop:finite_time}.

\subsection{Mild formulation}
\label{S:mild_formulation}
\begin{prop}\label{prop:termes_sys_micros} 
The process $\left(Y_N(t)\right)_{t\geq 0}$ satisfies  the following semimartingale decomposition in $D([0,T],L^2(I))$, written in a mild form: for any $0\leq t_0\leq t$
\begin{equation}\label{eq:def_Y_N_termes}
Y_N(t)=e^{(t-t_0)\mathcal{L}}Y_N(t_0) + \phi_N(t_0,t) + \zeta_N(t_0,t)
\end{equation} 
where:
\begin{equation}\label{eq:def_phi_N}
\phi_N(t_0,t)=\int_{t_0}^t e^{(t-s)\mathcal{L}}r_N(s)ds
\end{equation}
with

\begin{multline}\label{eq:def_r_N}
r_N(t)(x)=T_W\left( g_N(t)\right)(x)+\\
\sum_{i=1}^N \left( \dfrac{1}{N\rho_N} \sum_{j=1}^N  \xi_{ij}^{(N)}F(X_{N,j }(t),\eta_t(x_j)) - \int_I W(x,y) F(X_N(t,y),\eta_t(y))dy \right)\mathbf{1}_{B_{N,i}}(x),
\end{multline}

\begin{multline}\label{eq:def_gN(s)}
g_N(t)(y):=  \int_0^1 (1-r) \partial^2_x F\left( X_\infty(y)+rY_N(t)(y),(1-r)\eta_\infty(y)+r\eta_t(y)\right) Y_N(t)(y)^2  dr+\\
\int_0^1 (1-r)\left(\eta_t(y)-\eta_\infty(y)\right)\cdot\partial^2_\eta F\left( X_\infty(y)+r Y_N(t)(y),(1-r)\eta_\infty(y)+r\eta_t(y)\right) \left(\eta_t(y)-\eta_\infty(y)\right)  dr\\
+\int_0^1 2(1-r) \partial^2_{x,\eta}F\left(X_\infty(y)+rY_N(t)(y),(1-r)\eta_\infty(y)+r\eta_t(y)\right)\cdot\left(\eta_t(y)-\eta_\infty(y)\right)Y_N(t)(y)dr\\
+\partial_\eta F\left(X_\infty(y),\eta_\infty(y)\right)\cdot \left(\eta_t(y) - \eta_\infty(y)\right),
\end{multline}
and
\begin{equation}\label{eq:def_zeta_N}
\zeta_N(t_0,t)=\int_{t_0}^t e^{(t-s)\mathcal{L}}dM_N(s)
\end{equation}
with
\begin{equation}\label{eq:def_M_N}
M_N(t)= \sum_{i=1}^N \sum_{j=1}^N \dfrac{w_{ij}}{N} \left( Z_{N,j}(t) - \int_0^t\lambda_{N,j}(s)ds\right) \mathbf{1}_{B_{N,i}}.
\end{equation}
\end{prop}
$\phi_N$ is the drift term and $\zeta_N$ is the noise term coming from the jumps of the process $X_N$.

\begin{proof}[Proof of Proposition~\ref{prop:termes_sys_micros}]
From \eqref{eq:def_UiN} and \eqref{eq:def_UN}, we obtain that $X_N$ verifies
\begin{equation}\label{eq:dUN}
dX_N(t)=-\alpha X_N(t)dt + \sum_{i=1}^N  \sum_{j=1}^N \dfrac{w_{ij}}{N}  dZ_{N,j}(t)\mathbf{1}_{ B_{N,i}}.
\end{equation}
The centered noise  $M_N$ defined in \eqref{eq:def_M_N} verifies $$\displaystyle dM_{N}(t):= \sum_{i=1}^N \sum_{j=1}^N \dfrac{w_{ij}}{N} \left( dZ_{N,j}(t) - F(X_{N,j}(t),\eta_t(x_j))dt\right) \mathbf{1}_{B_{N,i}},$$ and is a martingale in $L^2(I)$. Thus recalling the definition of $X_\infty$ in \eqref{eq:def_u_infty} and by inserting the term $\sum_{i=1}^N \sum_{j=1}^N \dfrac{w_{ij}}{N}F(X_{N,j}(t),\eta_t(x_j))dt\mathbf{1}_{ B_{N,i}}$ in \eqref{eq:dUN}, we obtain 
\begin{equation}\label{eq:partial_incomplet}
d Y_N(t) = -\alpha Y_N(t)  + d M_{N}(t) + \sum_{i=1}^N \left( \sum_{j=1}^N \dfrac{w_{ij}}{N}  F(X_{N,j}(t),\eta_t(x_j))\right)\mathbf{1}_{B_{N,i}}dt  - T_W F(X_\infty,\eta_\infty)dt.
\end{equation}
A Taylor's expansion gives
$$F(X_N(t,y),\eta_t(y)) - F(X_\infty(y),\eta_\infty(y))=\partial_x F\left(X_\infty(y),\eta_\infty(y)\right) \left( X_N(t,y) - X_\infty(y)\right) + g_N(t)(y),$$
with $g_N$ given in \eqref{eq:def_gN(s)}. Hence,
we have with $G$ defined in \eqref{eq:def_G}
$$-T_W F(X_\infty,\eta_\infty)(x)=-\int_I W(x,y)F(X_N(t,y),\eta_t(y))dy + T_W(GY_N(t)) + T_Wg_N(t)(x),$$ 
hence coming back to \eqref{eq:partial_incomplet} and recognizing the operator $\mathcal{L}$ \eqref{eq:def_operator_L}
\begin{multline*}
d Y_N(t) = \mathcal{L} Y_N(t)  + d M_{N}(t) + \sum_{i=1}^N \left( \sum_{j=1}^N \dfrac{w_{ij}}{N}  F(X_{N,j}(t),\eta_t(x_j))\right)\mathbf{1}_{B_{N,i}}dt\\ - T_W F(X_N(t,\cdot),\eta_t(\cdot))dt+ T_Wg_N(t).
\end{multline*}
We recognize $r_N$ defined in \eqref{eq:def_r_N}, and obtain exactly 
\begin{equation}\label{eq:def_Y_N_diff}
dY_N(t) = \mathcal{L}Y_N(t)dt + r_N(t)dt + dM_N(t).
\end{equation}
Then the mild formulation \eqref{eq:def_Y_N_termes} is a direct consequence of Lemma 3.2 of \cite{Zhu2017}: the unique strong solution to \eqref{eq:def_Y_N_diff} is indeed given by \eqref{eq:def_Y_N_termes}.
\end{proof}

\begin{prop}[Noise perturbation]\label{prop:noise_perturbation} Let $m\geq 1$ and $T> t_0\geq 0$. Under Hypotheses \ref{hyp_globales} and \ref{hyp:scenarios}, there exists a constant $C=C(T,m,F,\eta_0)>0$ such that $\mathbb{P}$-almost surely for $N$ large enough: 
$$\mathbf{E}\left[\sup_{s\leq T} \Vert \zeta_N(t_0,s) \Vert_2^{2m}\right] \leq \dfrac{C}{\left(N\rho_N\right)^m}.$$
\end{prop}
 The proof is postponed to Section \ref{S:proof_NP}.

\begin{prop}[Drift term]\label{prop:drift_term}
Under Hypothesis \ref{hyp_globales}, for any $t\geq t_0>0$, $\mathbb{P}$-almost surely if $N$ is large enough,
\begin{align}\label{eq:control_drift}
\Vert \phi_N(t_0,t)\Vert_2 &\leq C_\text{drift} \left( \int_{t_0}^t e^{-(t-s)\gamma} \Vert Y_N(s) \Vert_2^2ds + G_{N}+\int_{t_0}^t e^{-\gamma(t-s)} \left(\delta_s^2 +\delta_s\right) ds\right),
\end{align}
where $C_\text{drift}=C_{W,F,\alpha}$, $\gamma$ is defined in \eqref{eq:def_gamma}, $\delta_s$ is defined in \eqref{eq:def_delta_s} and $G_N=G_N(\xi)$ is an explicit quantity  to be found in the proof that tends to 0 as $N\to \infty$.
\end{prop}
 The proof is postponed to Section \ref{S:proof_drift}.


\subsection{Proof of the large time behaviour}
\label{S:proof_mainTh}
We prove here Theorem~\ref{thm:long_time}, based on the results of Section~\ref{S:mild_formulation}. The approach followed is somehow formally similar to the strategy of proof developed in \cite{Coppini2022} for the diffusion case.
\begin{proof}[Proof of Theorem \ref{thm:long_time}]
Choose $m\geq 1$ and $t_f>0$. Let 
\begin{equation}\label{eq:def_varepsilon_0}
\varepsilon_0 = \dfrac{\gamma}{6 C_{drift}},
\end{equation}
where $\gamma$ is defined in \eqref{eq:def_gamma} and the constant $C_{drift}$ comes from Proposition \ref{prop:drift_term} above. We consider $\varepsilon$ small enough, such that $\varepsilon<\varepsilon_0$.
As $(X_t)$ converges uniformly towards $X_\infty$ (Theorem \ref{thm:large_time_cvg_u_t}), there exists $t_\varepsilon^1<\infty$ such that 
\begin{equation}\label{eq:choice_t_varepsilon1}
\left\Vert X_t-X_\infty \right\Vert_2\leq \dfrac{\varepsilon}{4},\quad t\geq t_\varepsilon^1.
\end{equation}Moreover, with \eqref{eq:def_delta_s}, we also have that $\int_0^t e^{-\gamma(t-s)}\left(\delta_s^2+\delta_s\right)ds \xrightarrow[t\to\infty]{}0$, hence there exists $t_\varepsilon^2<\infty$ such that 
\begin{equation}\label{eq:choice_t_varepsilon2}
C_\text{drift}\int_0^t e^{-\gamma(t-s)}\left(\delta_s^2+\delta_s\right)ds \leq \dfrac{\varepsilon}{18}, \quad  t\geq t_\varepsilon^2.
\end{equation}
We set now $t_\varepsilon=\max(t_\varepsilon^1,t_\varepsilon^2)$. Let  $T$ such that
\begin{equation}\label{eq:def_T}
e^{-\gamma T} <\frac{1}{3}, \quad T>t_f.
\end{equation}
The strategy of proof relies on the following time discretisation. The point is to control $\Vert X_N(t)-X_\infty\Vert_2$ on $[t_\varepsilon,T_N]$ with
\begin{equation}\label{eq:def_TN}
T_N:=a_N T+t_\varepsilon, \quad \text{with } a_N:=\lceil (N\rho_N)^m \rceil,
\end{equation} 
which will imply the result \eqref{eq:long_time_pol} as $[t_\varepsilon,(N \rho_{ N})^{ m}t_f]\subset [t_\varepsilon,T_N]$ since $T>t_f$. We decompose below the interval $[t_\varepsilon,T_N]$ into $a_N$ intervals of length $T$. We define the following events, with $0\leq t_a\leq t_b$ (recall that $Y_N(t)=X_N(t)-X_\infty$)
\begin{align}\label{eq:def_events}
A_1^N(\varepsilon)&:=\left\{ \left\Vert X_N(t_\varepsilon)-X_\infty\right\Vert_2 \leq \dfrac{\varepsilon}{2}\right\},\\
A_2^N(\varepsilon)&:=\left\{ \sup_{t\in [t_\varepsilon,t_\varepsilon+T]} \left\Vert \zeta_N(t_\varepsilon, t) \right\Vert_2 \leq\dfrac{\varepsilon}{18}\right\},\\
E(t_a,t_b)&:= \left\{ \max \left( 2\left\Vert Y_N(t_a)\right\Vert_2, \sup_{t\in [t_a,t_b]} \left\Vert Y_N(t) \right\Vert_2, 2\left\Vert Y_N(t_b)\right\Vert_2 \right) \leq \varepsilon\right\}\label{eq:def_event_E}.
\end{align}

By \eqref{eq:choice_t_varepsilon1}, and as Proposition \ref{prop:finite_time} gives that $\mathbf{P}\left( \sup_{t\in [0,t_\varepsilon]} \left\Vert X_N(t) - X_t\right\Vert_2>\dfrac{\varepsilon}{4}\right) \xrightarrow[N\to\infty]{}0
$, we have by triangle inequality 
\begin{equation}\label{eq:A1P1}
\mathbf{P}\left(  A_1^N(\varepsilon) \right) \xrightarrow[N\to\infty]{}1.
\end{equation}

\paragraph{Step 1}   We have from the definition \eqref{eq:def_event_E} of $E(t_a,t_b)$ that
\begin{equation}\label{eq:PminE}
\mathbf{P}\left( \sup_{t\in [t_\varepsilon,T_N]} \left\Vert X_N(t) - X_\infty\right\Vert_2\leq\varepsilon\right) \geq \mathbf{P}\left( E(t_\varepsilon,T_N)\right) = \mathbf{P}\left( E(t_\varepsilon, T_N)\vert A_1^N(\varepsilon) \right) \mathbf{P}\left(A_1^N(\varepsilon)\right).
\end{equation}
Moreover,
\begin{align*}
&\mathbf{P}\left( E(t_\varepsilon,T_N)\vert A_1^N(\varepsilon)\right)\\
&= \mathbf{P}\left( E(t_\varepsilon,t_\varepsilon+a_N T)\vert A_1^N(\varepsilon)\right)\\
&\geq  \mathbf{P}\left( E(t_\varepsilon,t_\varepsilon+a_N T)\cap E(t_\varepsilon,t_\varepsilon+(a_N-1) T) \vert A_1^N(\varepsilon)\right)\\
&=  \mathbf{P}\left( E(t_\varepsilon,t_\varepsilon+a_N T)\vert E(t_\varepsilon,t_\varepsilon+(a_N-1) T) \cap A_1^N(\varepsilon)\right)\mathbf{P}\left(  E(t_\varepsilon,t_\varepsilon+(a_N-1) T) \vert A_1^N(\varepsilon)\right).
\end{align*}
Recall that we are in the exponential case \eqref{eq:def_exponential}, so that $\left( X_N(t)\right)_t$ is a Markov process. Thus by Markov property
\begin{align*}
&\mathbf{P}\left( E(t_\varepsilon,t_\varepsilon+a_N T)\vert E(t_\varepsilon,t_\varepsilon+(a_N-1) T) \cap A_1^N(\varepsilon)\right)\\
=& \mathbf{P}\left( E(t_\varepsilon+(a_N-1)T,t_\varepsilon+a_N T)\vert E(t_\varepsilon,t_\varepsilon+(a_N-1) T)\right)\\
=&  \mathbf{P}\left( E(t_\varepsilon+(a_N-1)T,t_\varepsilon+a_N T)\left\vert \left\{ \left\Vert Y_N(t_\varepsilon+(a_N-1)T)\right\Vert_2 \leq \dfrac{\varepsilon}{2}\right\}\right.\right).
\end{align*}
$\mathbf{P}\left( E(t_\varepsilon+(a_N-1)T,t_\varepsilon+a_N T)\vert \left\{ \left\Vert Y_N(t_\varepsilon+(a_N-1)T)\right\Vert_2 \leq \dfrac{\varepsilon}{2}\right\}\right))$ means that, under an initial condition at $t_\varepsilon+(a_N-1)T$, we look at the probability that $Y_N$ stays under $\varepsilon$ on the interval $[t_\varepsilon+(a_N-1)T,t_\varepsilon+a_N T]$ of size $T$ and comes back under $\dfrac{\varepsilon}{2}$ at the final time $t_\varepsilon+a_N T$. By Markov's property, it is exactly $\mathbf{P}\left( E(t_\varepsilon,t_\varepsilon+T) \vert A_1^N(\varepsilon)\right)$. An immediate iteration gives then
\begin{equation}\label{eq:Ean}
\mathbf{P}\left( E(t_\varepsilon,T_N)\vert A_1^N(\varepsilon)\right) \geq \mathbf{P}\left( E(t_\varepsilon,t_\varepsilon+T) \vert A_1^N(\varepsilon)\right)^{a_N}.
\end{equation}
By \eqref{eq:A1P1}, from now on we consider that we are on this event $A_1^N(\varepsilon)$ and omit this notation for simplicity.

\paragraph{Step 2} We show that 
\begin{equation}\label{eq:AN2incluE}
A_2^N(\varepsilon)\subset E(t_\varepsilon, t_\varepsilon+T).
\end{equation} Let us place ourselves in $A_2^N(\varepsilon)$. As we are also under $A_1^N(\varepsilon)$, we have indeed $\left\Vert Y_N(t_\varepsilon)\right\Vert_2\leq \dfrac{\varepsilon}{2}$ for the first condition of $E(t_\varepsilon, t_\varepsilon+T)$. As $Y_N$ verifies \eqref{prop:termes_sys_micros}, it can be written for $t\geq t_\varepsilon$
\begin{equation}\label{eq:Y_N-MF_t_epsilon}
Y_N(t)=e^{\mathcal{L}(t-t_\varepsilon)}Y_N(t_\varepsilon) +  \phi_N(t_\varepsilon, t)+\zeta_N(t_\varepsilon, t).
\end{equation}

For any $t\in [t_\varepsilon,t_\varepsilon+T]$,
\begin{align}\label{eq:maintheorem_aux2}
\left\Vert \phi_N(t_\varepsilon, t) \right\Vert_2 &\leq C_\text{drift} \left( \int_{t_\varepsilon}^t e^{-(t-s)\gamma} \Vert Y_N(s) \Vert_2^2ds + G_{N}+\int_{t_0}^t e^{-\gamma(t-s)} \left(\delta_s^2 +\delta_s\right) ds\right)\notag\\
&\leq C_\text{drift} \left( \int_{t_\varepsilon}^t e^{-(t-s)\gamma} \Vert Y_N(s) \Vert_2^2ds \right)+ \dfrac{\varepsilon}{9}
\end{align}
where the first inequality comes from Proposition  \ref{prop:drift_term}, and the second is true for $N$ large enough using $G_N\to 0$ and \eqref{eq:choice_t_varepsilon2}. Coming back to \eqref{eq:Y_N-MF_t_epsilon}, using that by Proposition \ref{prop:operateur_L} 
\begin{equation}\label{eq:maintheorem_aux4}
\left\Vert e^{\mathcal{L}(t-t_\varepsilon)}Y_N(t_\varepsilon) \right\Vert_2\leq  e^{-\gamma(t-t_\varepsilon)}\left\Vert Y_N(t_\varepsilon)\right\Vert_2,
\end{equation}
and using \eqref{eq:maintheorem_aux2}, we have on $A_1^N(\varepsilon)\cap A_2^N(\varepsilon)$
$$\left\Vert Y_N(t) \right\Vert_2 \leq \dfrac{\varepsilon}{2} +  C_\text{drift} \left( \int_{t_\varepsilon}^t e^{-(t-s)\gamma} \Vert Y_N(s) \Vert_2^2ds \right)+\dfrac{\varepsilon}{9}+\dfrac{\varepsilon}{18}.$$
Let $\delta>0$ such that $\delta\leq \min\left( \dfrac{\varepsilon}{6},\dfrac{\gamma}{9C_\text{drift}}\right)$. Recall that $\left\Vert Y_N(\cdot)\right\Vert_2$ is not a continuous function, it jumps whenever a spike of the process $\left(Z_{N,1},\cdots,Z_{N,N}\right)$ occurs, but the size jump never exceeds $\dfrac{1}{N}$, and for $N$ large enough $\dfrac{1}{N}\leq \dfrac{\delta}{2}$. Then, one can apply Lemma \ref{lem:gronwal_quadratic} and obtain that for all $N$ large enough,
\begin{equation}\label{eq:maintheorem_aux3}
\sup_{t\in [t_\varepsilon,t_\varepsilon+T]}\left\Vert Y_N(t)\right\Vert_2\leq \dfrac{\varepsilon}{2}+ 3\delta\leq\varepsilon.
\end{equation}
It remains to prove that $\left\Vert Y_N(t_\varepsilon+T)\right\Vert_2\leq\dfrac{\varepsilon}{2}$. We obtain from \eqref{eq:Y_N-MF_t_epsilon}, \eqref{eq:maintheorem_aux2} and \eqref{eq:maintheorem_aux4} for $t=t_\varepsilon + T$ on $A_1^N(\varepsilon)\cap A_2^N(\varepsilon)$
$$\left\Vert Y_N(t_\varepsilon + T)\right\Vert_2\leq  e^{-\gamma T}\dfrac{\varepsilon}{2} + \dfrac{\varepsilon}{6} + C_\text{drift}\int_{t_\varepsilon}^{t_\varepsilon+T} e^{-(t_\varepsilon+T-s)\gamma} \Vert Y_N(s) \Vert_2^2ds.$$
Using the a priori bound \eqref{eq:maintheorem_aux3}
\begin{align*}
\left\Vert Y_N(t_\varepsilon + T)\right\Vert_2&\leq  e^{-\gamma T}\dfrac{\varepsilon}{2} +\dfrac{\varepsilon}{12} + \varepsilon^2 \dfrac{C_\text{drift}}{\gamma}\leq e^{-\gamma T}\dfrac{\varepsilon}{2} +\dfrac{\varepsilon}{6} + \dfrac{\varepsilon}{6}\leq \dfrac{\varepsilon}{2},
\end{align*}
where we recall the particular choices of $T$ and $\varepsilon<\varepsilon_0$ in \eqref{eq:def_T} and \eqref{eq:def_varepsilon_0}. This concludes the proof of \eqref{eq:AN2incluE}.

\paragraph{Step 3} We obtain with \eqref{eq:Ean}
 and Markov's inequality,
\begin{align*}
\mathbf{P}\left( E(t_\varepsilon,T_N)\right)&\geq  \mathbf{P}\left( E(t_\varepsilon,t_\varepsilon+T)\right)^{a_N}\geq \mathbf{P}(A_2^N(\varepsilon))^{a_N}\\
&=\left( 1 - \mathbf{P}\left( \sup_{t\in [t_\varepsilon,t_\varepsilon+T]} \left\Vert \zeta_N(t_\varepsilon, t) \right\Vert_2 >\dfrac{\varepsilon}{18}\right) \right)^{a_N}\\
&\geq \left( 1 -  18^{2m'}\dfrac{\mathbb{E}\left[ \sup_{t\in [t_\varepsilon,t_\varepsilon+T]} \left\Vert \zeta_N(t_\varepsilon, t) \right\Vert_2^{2m'}\right] }{\varepsilon^{2m'}}\right)^{a_N},
\end{align*}
where we have taken $m'>m$.
With Proposition \ref{prop:noise_perturbation}, it gives
$$\mathbf{P}\left( E(t_\varepsilon,T_N)\right)\geq \left(1-  \dfrac{C}{\left(\varepsilon^2N\rho_N\right)^{m'}}\right)^{a_N}=\exp\left( a_N \ln \left(1-  \dfrac{C}{\left(\varepsilon^2N\rho_N\right)^{m'}}\right)\right).$$
By definition \eqref{eq:def_TN}, $a_N=o\left(N\rho_N \right)^{m'}$, the right term tends to 1 as $N$ goes to $\infty$ under Hypothesis \ref{hyp:scenarios}. By \eqref{eq:PminE}, we conclude that
$$\mathbf{P}\left( \sup_{t\in [t_\varepsilon,T_N]} \left\Vert X_N(t) - X_\infty\right\Vert_2\leq\varepsilon\right) \xrightarrow[N\to\infty]{}1.$$
This concludes the proof of Theorem~\ref{thm:long_time}.
\end{proof}

\section{Proofs - Noise perturbation}\label{S:proof_NP}
In this section, we prove Proposition~\ref{prop:noise_perturbation} concerning the control of the noise perturbation $\zeta_N(t_0,t)$ defined in \eqref{eq:def_zeta_N}. For simplicity of notation, we assume that $t_0=0$. Recall the expression of $\left(Z_{N,j}\right)_{1\leq j \leq N}$ in \eqref{eq:def_ZiN}. Introduce the compensated measure $\tilde{\pi}_j(ds,dz):=\pi_j(ds,dz)-\lambda_{N,j}dsdz$, so that with the linearity of $(e^{t\mathcal{L}})_{t\geq 0}$, we obtain that $\zeta_N$ can be written as
\begin{equation}\label{eq:zeta_N_chi}
 \zeta_N(0,t) = \sum_{j=1}^N \int_0^t\int_0^\infty e^{(t-s)\mathcal{L}}\chi_j(s,z) \tilde{\pi}_j(ds,dz),
\end{equation}
with $\displaystyle \chi_j(s,z):=\left( \sum_{i=1}^N \mathbf{1}_{B_{N,i}} \dfrac{w_{ij}}{N} \mathbf{1}_{z\leq \lambda_{N,j}(s)}\right)$. The proof of Proposition \ref{prop:drift_term} relies on a adaptation of an argument given in \cite{Zhu2017} (Theorem 4.3), where a similar quantity to \eqref{eq:zeta_N_chi} is considered for $N=1$.

\subsection{Control of the moments of the process $Z_{N, i}$}

\begin{prop}\label{prop:control_mean_Zj^m}
 Let $m\geq 1$ and $T>0$.  Under Hypotheses \ref{hyp_globales} and \ref{hyp:scenarios}, $\mathbb{P}$-almost surely
$$\sup_{N\geq 1} \mathbf{E}\left[\dfrac{1}{N} \sum_{j=1}^N Z_{N,j}(T)^m \right] <\infty.$$
\end{prop}

\begin{proof}
Let $N\geq 1$. We have for any $i\in \llbracket 1,N\rrbracket$
\begin{align}
\mathbf{E} \left[ Z_{N,i}(T)^m \right] &\leq \mathbf{E} \left[ \left( \left( Z_{N,i}(T))-\int_0^T \lambda_{N,i}(t)dt\right) + \int_0^T \lambda_{N,i}(t)dt\right)^m \right] \notag\\
&\leq 2^{m-1}\mathbf{E} \left[ \left( Z_{N,i}(T)-\int_0^T \lambda_{N,i}(t)dt\right)^m \right] + 2^{m-1} \mathbf{E} \left[ \left(\int_0^T \lambda_{N,i}(t)dt\right)^m \right]\notag \\
&\leq 2^{m-1}C \mathbf{E} \left[ \left(\int_0^T \lambda_{N,i}(t)dt\right)^{\frac{m}{2}} \right] + (2T)^{m-1}  \mathbf{E} \left[ \int_0^T {\lambda_{N,i}}(t)^mdt \right], \label{eq:Z_iNT_m}
\end{align}
where we used Jensen's inequality and Burkholder-Davis-Gundy Inequality on the martingale $\left( Z_{N,i}(T)-\int_0^T \lambda_{N,i}(t)dt\right)$. Similarly, we obtain 
$$ \mathbf{E} \left[ \left(\int_0^T \lambda_{N,i}(t)dt\right)^{\frac{m}{2}} \right] \leq T^{\frac{m}{2}-1} \mathbf{E} \left[ \int_0^T \left(\lambda_{N,i}(t)\right)^{\frac{m}{2}}	dt \right].$$ We focus now on the term $ \mathbf{E} \left[ \int_0^T \lambda_{N,i}(t)^kdt \right]$ for $k\geq 1$. From the definition of $\lambda_{N,i}$ \eqref{eq:def_lambdaiN_intro}, by Lipschitz continuity of $F$ and with Jensen's inequality
\begin{multline*}
 \mathbf{E} \left[ \int_0^T \lambda_{N,i}(t)^kdt \right] \leq  2^{k-1} T F(0,\eta_t(x_i))^k\\ + 2^{k-1} \Vert F \Vert_L^k \mathbf{E} \left[\int_0^T \left( \dfrac{1}{N} \sum_{j=1}^N \int_0^{t-}w_{ij} e^{-\alpha(t-s)} dZ_{N,j}(s)\right)^k dt\right].
\end{multline*}
Let $S_i:=\sum_{j=1}^N \dfrac{w_{ij}}{N}$. By \eqref{eq:estimees_IC}, we have that $\mathbb{P}$-almost surely, $\limsup_{N\to\infty}\sup_{1\leq i \leq N} S_i \leq 2$. We obtain with discrete Jensen's inequality that for any $t\geq 0$
$$\left(\dfrac{1}{N} \sum_{j=1}^N \int_0^{t-}w_{ij} e^{-\alpha(t-s)}  dZ_{N,j}(s)\right)^k \leq S_i^k\left( \sum_{j=1}^N \dfrac{w_{ij}}{NS_i}Z_{N,j}(t)\right)^k\leq S_i^{k-1} \sum_{j=1}^N \dfrac{w_{ij}}{N} Z_{N,j}(t)^k.$$
We obtain then
$$\mathbf{E} \left[ \int_0^T \lambda_{N,i}(t)^kdt \right] \leq   C_{T,F,\eta_0,k} + C_{k,F} \sum_{j=1}^N \dfrac{w_{ij}}{N} \mathbf{E}\left[\int_0^T Z_{N,j}(t)^{k}dt\right],$$
thus, going back to \eqref{eq:Z_iNT_m}, with $C=C_{T,F,\eta_0,m}$ 
\begin{align*}
\mathbf{E}\left[\dfrac{1}{N} \sum_{j=1}^N Z_{N,j}(T)^m  \right] &\leq \dfrac{C}{N} \sum_{i=1}^N   \left(  \mathbf{E} \left[ \int_0^T \lambda_{N,i}(t)^{\frac{m}{2}}dt \right] + C \mathbf{E} \left[ \int_0^T \lambda_{N,i}(t)^mdt \right]\right)\\
&\leq C\left(1 +  \sum_{i,j=1}^N \dfrac{w_{ij}}{N^2} \mathbf{E}\left[\int_0^T Z_{N,j}(t)^{\frac{m}{2}}dt\right] +  \sum_{i,j=1}^N \dfrac{w_{ij}}{N^2} \mathbf{E}\left[\int_0^T Z_{N,j}(t)^mdt\right]\right).
\end{align*}
With \eqref{eq:estimees_IC}, it gives that, $\mathbb{P}$-almost surely for $N$ large enough
\begin{align*}
\mathbf{E}\left[\dfrac{1}{N} \sum_{j=1}^N Z_{N,j}(T)^m  \right] \leq  C\left(1 +  \int_0^T \mathbf{E}\left[\dfrac{1}{N}\sum_{j=1}^N Z_{N,j}(t)^{\frac{m}{2}}\right]dt +  \int_0^T\mathbf{E}\left[ \dfrac{1}{N}\sum_{j=1}^N Z_{N,j}(t)^m\right]dt\right).
\end{align*}
As for any $t\geq 0$
$$\mathbf{E}\left[\dfrac{1}{N}\sum_{i=1}^N Z_{N,i}(t)\right] =  \dfrac{1}{N}\sum_{i=1}^N \mathbf{E}\left[ \int_0^t \lambda_{N,i}(s)ds\right] \leq C_{T,\eta_0,F} +  C_{T,\eta_0,F} \int_0^t\mathbf{E}\left[ \dfrac{1}{N}\sum_{j=1}^N Z_{N,j}(s)\right]ds,$$  Gr{\"o}nwall's lemma  gives that $\displaystyle\sup_{t\leq T} \mathbf{E}\left[\dfrac{1}{N}\sum_{i=1}^N Z_{N,i}(t)\right]  <\infty$ (independently of $N$) and similarly an immediate iteration gives that for any $k\geq 0$, $ \displaystyle\sup_{N\geq 1}\mathbf{E} \left[ \dfrac{1}{N} \sum_{j=1}^N Z_{N,j}(T)^{2^k} \right]<\infty$ which concludes the proof.
\end{proof}
 
\subsection{Proof of Proposition \ref{prop:noise_perturbation}}
\begin{proof}
We divide the proof in different steps. Fix $m\geq 1$. We prove Proposition \ref{prop:noise_perturbation} for the choice $t_0=0$, but it remains the same for a general initial time $t_0\geq 0$.

\textit{Step 1 -} The functional $\phi:L^2(I)\to \mathbb{R}$ given by $\phi(v)=\Vert v \Vert_2^{2m}$ is of class $\mathcal{C}^2$ (recall that $\zeta_N\in L^2(I)$) so that by Itô formula on the expression \eqref{eq:zeta_N_chi} we obtain
\begin{align}\label{eq:zeta_spatial_def}
\phi\left(\zeta_N(t)\right)&= \int_0^t \phi'\left(\zeta_N(s)\right) \mathcal{L}\left(\zeta_N(s)\right)ds + \sum_{j=1}^N \int_0^t \int_0^\infty \phi'\left(\zeta_N(s-)\right)\chi_j(s,z)\tilde{\pi}_j(ds,dz) \notag\\+ &\sum_{j=1}^N\int_0^t\int_0^\infty \left[ \phi\left( \zeta_N(s-)+\chi_j(s,z)\right) - \phi\left(\zeta_N(s-)\right) - \phi'\left(\zeta_N(s-)\right)\chi_j(s,z)\right]\pi_j(ds,dz)\notag\\
&:= I_0(t) + I_1(t) + I_2(t).
\end{align}
We have then for any $v,h,k\in L^2(I)$, $\phi'(v)h=2m\Vert v \Vert_2^{2m-2}\text{Re}\left(\langle v,h\rangle\right)\in\mathbb{R}$ and $\phi''(v)(h,k)=2m(2m-1)\Vert v \Vert_2^{2m-4}\text{Re}\langle v,k \rangle \text{Re}\langle v,h \rangle +2m\Vert v \Vert^{2m-2} \text{Re}\langle h,k\rangle$.

\medskip
\textit{Step 2 -} We have $I_0(t)= \int_0^t 2m \Vert \zeta_N(s) \Vert_2^{2m-2}\text{Re}\left( \langle \zeta_N(s),\mathcal{L}(\zeta_N(s))\rangle \right)ds$.  From Proposition \ref{prop:operateur_L}, $\mathcal{L}$ generates a contraction semi-group hence for any $s\geq 0$, $\text{Re}\left( \langle \zeta_N(s),\mathcal{L}(\zeta_N(s))\right)\rangle\leq 0$ by Lumer-Philipps Theorem (see Section 1.4 of \cite{Pazy1974}). Then for any $t\geq 0$ we have $I_0(t)\leq 0$. 

\medskip
\textit{Step 3 -} About $I_1$ in \eqref{eq:zeta_spatial_def}, with $\alpha_j(s,z):=2m \Vert \zeta_N(s-)\Vert_2^{2m-2} \langle \zeta_N(s-),\chi_j(s,z)\rangle\in \mathbb{R}$,
$$ I_1(t)= \sum_{j=1}^N \int_0^t\int_0^\infty \alpha_j(s,z) \tilde{\pi}_j(ds,dz).$$ 
$I_1$ is then a real martingale. By Burkholder-David-Gundy inequality, there exists a constant $C>0$ such that for any $t\geq 0$:
$$\mathbf{E}\left[ \sup_{s\leq t} \left|I_1(s)\right|\right] \leq C \mathbf{E}\left[ \sqrt{[I_1]_t} \right],$$
where $[I_1]_t=\sum_{s\leq t} \left| \Delta I_1(s) \right|^2$ stands for the quadratic variation of $I_1$. It is computed as follows (as the $(\pi_j)_{1\leq j \leq N}$ are independent, there are almost surely no simultaneous jumps so that $[\tilde{\pi}_j,\tilde{\pi}_{j'}]=0$ if $j\neq j'$):
\begin{align*}
[I_1]_t
&=  \sum_{j=1}^N\int_{0}^t \int_0^\infty  \alpha_j(s,z)^2\pi_{j}(ds,dz)\\
&=  \sum_{j=1}^N\int_{0}^t \int_0^\infty  \left(2m \Vert \zeta_N(s-)\Vert_2^{2m-2} \langle \zeta_N(s-),\chi_j(s,z)\rangle\right)^2\pi_{j}(ds,dz)\\
&\leq 4m^2 \sup_{0\leq s \leq t} \left( \Vert \zeta_N(s)\Vert_2^{4m-2}\right) \sum_{j=1}^N\int_0^t\int_0^\infty \Vert \chi_j(s,z)\Vert_2^2 \pi_j(ds,dz).
\end{align*}
We obtain then 
$$\mathbf{E}\left[ \sqrt{[I_1]_t} \right] \leq 2m \mathbf{E}\left[ \sup_{0\leq s \leq t} \left( \Vert \zeta_N(s)\Vert_2^{2m-1}\right) \left(\sum_{j=1}^N\int_0^t\int_0^\infty \Vert \chi_j(s,z)\Vert_2^2 \pi_j(ds,dz)\right)^{\frac{1}{2}}\right].$$
Applying H{\"o}lder inequality with parameter $\frac{2m-1}{2m}+\frac{1}{2m}=1$ for the random variables $\sup_{0\leq s \leq t} \left( \Vert \zeta_N(s)\Vert_2^{2m-1}\right)$ and $ \left(\sum_{j=1}^N\int_0^t\int_0^\infty \Vert \chi_j(s,z)\Vert_2^2 \pi_j(ds,dz)\right)^{\frac{1}{2}}$, we obtain that $\mathbf{E}\left[ \sqrt{[I_1]_t} \right]$ is upper bounded by
$$ 2m \left( \mathbf{E}\left[ \sup_{0\leq s \leq t} \left( \Vert \zeta_N(s)\Vert_2^{2m}\right) \right]\right)^{\frac{2m-1}{2m}}
\left(\mathbf{E}\left[\left(\sum_{j=1}^N\int_0^t\int_0^\infty \Vert \chi_j(s,z)\Vert_2^2 \pi_j(ds,dz)\right)^m\right]\right) ^{\frac{1}{2m}}.$$
Let $\varepsilon>0$ to be chosen later. From Young's inequality, for any $a,b \geq 0$, we can write $ab=\left( \varepsilon^{\frac{2m-1}{2m}}a\right) \left( \varepsilon^{\frac{-(2m-1)}{2m}}b\right)\leq \frac{2m-1}{2m}\left( \varepsilon^{\frac{2m-1}{2m}}a\right)^{\frac{2m}{2m-1}}+\frac{1}{2m}\left( \varepsilon^{\frac{-(2m-1)}{2m}}b\right)^{2m}=\frac{2m-1}{2m}\varepsilon a^{\frac{2m}{2m-1}}+\frac{1}{2m}\varepsilon^{-(2m-1)}b^{2m}$. This gives for the choice $a=\left( \mathbf{E}\left[ \sup_{0\leq s \leq t} \left( \Vert \zeta_N(s)\Vert_2^{2m}\right) \right]\right)^{\frac{2m-1}{2m}}$ and $b=\left(\mathbf{E}\left[\left(\sum_{j=1}^N\int_0^t\int_0^\infty \Vert \chi_j(s,z)\Vert_2^2 \pi_j(ds,dz)\right)^m\right]\right) ^{\frac{1}{2m}}$:
\begin{multline*}
\mathbf{E}\left[\sqrt{ [I_1]_t} \right]  \leq  (2m-1)  \varepsilon \mathbf{E}\left[ \sup_{0\leq s \leq t} \left( \Vert \zeta_N(s)\Vert_2^{2m}\right) \right] \\+\varepsilon^{-(2m-1)} \mathbf{E}\left[\left(\sum_{j=1}^N\int_0^t\int_0^\infty \Vert \chi_j(s,z)\Vert_2^2 \pi_j(ds,dz)\right)^m\right].
\end{multline*}
We have then shown that, for the constant $C$ given by Burkholder-Davis-Gundy Inequality,
\begin{multline}\label{eq:spatial_I1}
\mathbf{E}\left[ \sup_{s\leq T} \left|I_1(s)\right|\right] \leq C (2m-1) \varepsilon \mathbf{E}\left[ \sup_{0\leq s \leq T} \left( \Vert \zeta_N(s)\Vert_2^{2m}\right) \right]\\ + C \varepsilon^{-(2m-1)}\mathbf{E}\left[\left(\sum_{j=1}^N\int_0^T\int_0^\infty \Vert \chi_j(s,z)\Vert_2^2 \pi_j(ds,dz)\right)^m\right].
\end{multline}

Let us focus now on $I_2$ in \eqref{eq:zeta_spatial_def}:
$$I_2(t)=\sum_{j=1}^N\int_0^t\int_0^\infty \left[ \phi\left( \zeta_N(s-)+\chi_j(s,z)\right) - \phi\left(\zeta_N(s-)\right) - \phi'\left(\zeta_N(s-)\right)\chi_j(s,z)\right]\pi_j(ds,dz).$$
For any jump $(s,z)$ of the Poisson measure $\pi_j$, from Taylor's Lagrange formula there exists $\tau_s\in (0,1)$ such that
\begin{multline*}
\phi\left( \zeta_N(s-)+\chi_j(s,z)\right) - \phi\left(\zeta_N(s-)\right) - \phi'\left(\zeta_N(s-)\right)\chi_j(s,z)\\= \dfrac{1}{2} \phi''\left(\zeta_N(s-)+\tau_s \chi_j(s,z)\right) \left( \chi_j(s,z),\chi_j(s,z) \right).
\end{multline*}
As $\phi''(v)(h,k)=2m(2m-1)\Vert v \Vert_2^{2m-4}\text{Re}\langle v,k \rangle \text{Re}\langle v,h \rangle +2m\Vert v \Vert^{2m-2} \text{Re}\langle h,k\rangle$ for any $v,h,k \in L^2(I)$, one has with Cauchy–Schwarz inequality that 
$$\phi''\left(\zeta_N(s-)+\tau_s \chi_j(s,z)\right) \left( \chi_j(s,z) \right)^2 \leq 4m^2 \Vert \zeta_N(s-)+\tau_s \chi_j(s,z) \Vert_2^{2m-2} \Vert \chi_j(s,z)\Vert_2^2.$$
But as $\Vert x+\tau y\Vert_2^2 \leq \max \left( \Vert x\Vert_2^2 , \Vert x+y \Vert_2^2\right)$ for any $x,y \in L^2(I)$ and $\tau \in (0,1)$, we have here
$$\Vert \zeta_N(s-)+\tau_s \chi_j(s,z) \Vert_2^{2m-2}  \leq \max\left( \Vert \zeta_N(s-)\Vert_2^{2m-2} , \Vert \zeta_N(s-)+ \chi_j(s,z) \Vert_2^{2m-2} \right),$$ and as
$\Vert \zeta_N(s-)\Vert_2^{2m-2} \leq \sup_{s\leq t} \Vert \zeta_N(s)\Vert_2^{2m-2}$ and $ \Vert \zeta_N(s-)+ \chi_j(s,z) \Vert_2^{2m-2} =  \Vert \zeta_N(s) \Vert_2^{2m-2} \leq \sup_{s\leq t} \Vert \zeta_N(s)\Vert_2^{2m-2}$, thus
$$\mathbf{E}\left[ \sup_{s\leq t}\vert I_2(s) \vert \right] \leq 2m^2   \mathbf{E}\left[ \sup_{s\leq t} \Vert \zeta_N(s)\Vert_2^{2m-2} \sum_{j=1}^N \int_0^t \int_0^\infty  \Vert \chi_j(s,z)\Vert_2^2\pi_j(ds,dz)\right].$$
We proceed now similarly as for $I_1$. From H{\"o}lder inequality, as $\frac{2m-2}{2m}+\frac{1}{m}=1$ we know that for any $A,B$ random non-negative variables,  $\mathbf{E}\left[AB\right] \leq \left( \mathbf{E}\left[A^{\frac{2m}{2m-2}}\right]\right)^{\frac{2m-2}{2m}} \left(  \mathbf{E}\left[ B^{m} \right]\right)^{\frac{1}{m}}$. It leads for the choice $A=\sup_{0\leq s \leq t} \left( \Vert \zeta_N(s)\Vert_2^{2m-2}\right)$ and $B  = \sum_{j=1}^N\int_0^t\int_0^\infty \Vert \chi_j(s,z)\Vert_2^2 \pi_j(ds,dz)$ to $\mathbf{E}\left[ \sup_{s\leq t}\vert I_2(s) \vert \right]$ equals to
$$2m^2   \left(\mathbf{E}\left[\sup_{0\leq s \leq t} \left( \Vert \zeta_N(s)\Vert_2^{2m}\right)\right]\right)^{\frac{2m-2}{2m}} \left(\mathbf{E}\left[\left(\sum_{j=1}^N\int_0^t\int_0^\infty \Vert \chi_j(s,z)\Vert_2^2 \pi_j(ds,dz)\right)^m\right]\right)^{\frac{1}{m}}.$$
With the same $\varepsilon$ introduced for $I_1$, from Young's inequality, for any $a,b \geq 0$, we can write $ab=\left( \varepsilon^{\frac{2m-2}{2m}}a\right) \left( \varepsilon^{\frac{-(2m-2)}{2m}}b\right)\leq \frac{2m-2}{2m}\left( \varepsilon^{\frac{2m-2}{2m}}a\right)^{\frac{2m}{2m-2}}+\frac{1}{m}\left( \varepsilon^{\frac{-(2m-2)}{2m}}b\right)^{m}=
\frac{2m-2}{2m}\varepsilon a^{\frac{2m}{2m-2}}+\frac{1}{m}\varepsilon^{-(2m-2)}b^{m}$. For the choice 
\begin{align*}
a&=\left( \mathbf{E}\left[ \sup_{0\leq s \leq t} \left( \Vert \zeta_N(s)\Vert_2^{2m}\right) \right]\right)^{\frac{2m-2}{2m}} \text{ and}\\
b&=\left(\mathbf{E}\left[\left(\sum_{j=1}^N\int_0^t\int_0^\infty \Vert \chi_j(s,z)\Vert_2^2 \pi_j(ds,dz)\right)^m\right]\right) ^{\frac{1}{m}},
\end{align*}
it gives that $\mathbf{E}\left[ \sup_{s\leq t}\vert I_2(s) \vert \right] $ is upper bounded by
\begin{align}\label{eq:spatial_I2}
& m(2m-2) \varepsilon \mathbf{E}\left[\sup_{0\leq s \leq t} \left( \Vert \zeta_N(s)\Vert_2^{2m}\right)\right] + 2m \varepsilon^{-(2m-2)}\mathbf{E}\left[\left(\sum_{j=1}^N\int_0^t\int_0^\infty \Vert \chi_j(s,z)\Vert_2^2 \pi_j(ds,dz)\right)^m\right].
\end{align}

Taking the expectation in \eqref{eq:zeta_spatial_def} and combining \eqref{eq:spatial_I1} and \eqref{eq:spatial_I2}, we obtain that
\begin{multline}\label{eq:spatial_zeta_maj_eps}
\mathbf{E}\left[\sup_{s\leq T} \Vert \zeta_N(s) \Vert_2^{2m}\right] \leq \varepsilon\left( C(2m-1)+m(2m-2) \right) \mathbf{E}\left[\sup_{0\leq s \leq T} \left( \Vert \zeta_N(s)\Vert_2^{2m}\right)\right] \\
\quad +  \left( C \varepsilon^{-(2m-1)}+2m \varepsilon^{-(2m-2)}\right) \mathbf{E}\left[\left(\sum_{j=1}^N\int_0^T\int_0^\infty \Vert \chi_j(s,z)\Vert_2^2 \pi_j(ds,dz)\right)^m\right].
\end{multline}

\medskip
\textit{Step 4 -} We can now fix $\varepsilon$ such that $\varepsilon\left( C(2m-1)+m(2m-2) \right) \leq \frac{1}{2}$ so that \eqref{eq:spatial_zeta_maj_eps} leads to 
\begin{equation}\label{eq:spatial_zeta_maj}
\mathbf{E}\left[\sup_{s\leq T} \Vert \zeta_N(s) \Vert_2^{2m}\right] \leq 2 C\mathbf{E}\left[\left(\sum_{j=1}^N\int_0^T\int_0^\infty \Vert \chi_j(s,z)\Vert_2^2 \pi_j(ds,dz)\right)^m\right],
\end{equation}
where $C>0$ depends only on $m$.

\medskip
\textit{Step 5 -} Let $A_N:=\mathbf{E}\left[\left(\sum_{j=1}^N\int_0^T\int_0^\infty \Vert \chi_j(s,z)\Vert_2^2 \pi_j(ds,dz)\right)^m\right]$. We have 
$$\Vert\chi_j(s,z)\Vert_2^2 = \int_I \left( \sum_{i=1}^N \mathbf{1}_{B_{N,i}}(x) \dfrac{w_{ij}}{N}\mathbf{1}_{z\leq \lambda_{N,j}(s)}\right)^2 dx = \mathbf{1}_{z\leq \lambda_{N,j}(s)}  \sum_{i=1}^N\dfrac{\xi_{ij}}{N^3\rho_N^2},
$$
which leads to, with the definition of $Z_{N,j}$ in \eqref{eq:def_ZiN}
\begin{align*}
A_N &= \mathbf{E}\left[\left(\sum_{i,j=1}^N\int_0^T \int_0^\infty\mathbf{1}_{z\leq \lambda_{N,j}(s)}  \dfrac{\xi_{ij}}{N^3\rho_N^2}\pi_j(ds,dz)\right)^m\right]\\
&\leq \left(\dfrac{1}{N\rho_N}\right)^m\mathbf{E}\left[ \left( \sum_{i,j=1}^N\dfrac{\xi_{ij}}{N^2\rho_N}Z_{N,j}(T)\right)^m\right].
\end{align*}
With \eqref{eq:estimees_IC}, Jensen's discrete inequality and \eqref{eq:spatial_zeta_maj}, it leads to
\begin{align*}
A_N &\leq \left(\dfrac{1}{N\rho_N}\right)^m\mathbf{E}\left[ \left( \sum_{j=1}^N  \dfrac{1}{N} \left( \sup_j \sum_{i=1}^N \dfrac{\xi_{ij}}{N\rho_N} \right) Z_{N,j}(T)\right)^m\right]\\
&\leq  \dfrac{C}{\left(N\rho_N\right)^m } \mathbf{E}\left[\dfrac{1}{N} \sum_{j=1}^N Z_{N,j}(T)^m  \right],
\end{align*}
hence the result with Proposition \ref{prop:control_mean_Zj^m}.
\end{proof}

\section{Proofs - Drift term}\label{S:proof_drift}

In this section, we prove Proposition~\ref{prop:drift_term} concerning the control of the drift term perturbation $\phi_N(t_0,t)$ defined in \eqref{eq:def_phi_N}.
\subsection{Notation}

We introduce the following auxiliary functions in $L^2(I)$:
\begin{align}
\Theta_{t,i,1} &:= \dfrac{1}{N\rho_N} \sum_{j=1}^N  \left(\xi_{ij}^{(N)}-\rho_NW(x_i,x_j)\right) F(X_{N,j }(t),\eta_t(x_j)),\label{eq:def_Theta1}\\
\Theta_{t,i,2} &:= \dfrac{1}{N} \sum_{j=1}^N  W(x_i,x_j)F(X_{N,j}(t),\eta_t(x_j)) - \int_I W(x_i,y)F(X_N(t,y),\eta_t(y))dy, \label{eq:def_Theta2}\\
\Theta_{t,i,3}(x) &:=  \int_I \left( W(x_i,y) - W(x,y) \right)F(X_N(t,y),\eta_t(y))dy, \label{eq:def_Theta3}.
\end{align}
From the expression of $r_N$ in \eqref{eq:def_r_N}, we have then 
\begin{equation}\label{eq:def_r_N_aux}
r_N(t)=\sum_{i=1}^N \left( \sum_{k=1}^3 \Theta_{t,i,k} \right)\mathbf{1}_{B_{N,i}} \\+ T_W\left( g_N(t) \right),
\end{equation}
and we can divide $\phi_N$ defined in \eqref{eq:def_phi_N} in several terms $\displaystyle \phi_N(t)=\sum_{k=0}^3 \phi_{N,k}(t)$ with
\begin{align}
\phi_{N,0}(t)&:= \int_{t_0}^t e^{(t-s)\mathcal{L}}T_W\left(g_N(s)\right)ds,\label{eq:def_phi_n0}\\
\phi_{N,k}(t)&:= \int_{t_0}^t e^{(t-s)\mathcal{L}}\sum_{i=1}^N \dfrac{1}{N} \Theta_{s,i,k}\mathbf{1}_{B_{N,i}}ds\quad \text{for }k\in \llbracket 1,3 \rrbracket, \label{eq:def_phi_nk}.
\end{align}

\subsection{Preliminary results}

\begin{lem}\label{lem:tilde_YN_control}
Denoting by $\tilde{Y}_N(s)(v):= Y_N(s)\left(\dfrac{\lceil Nv \rceil}{N}\right)$, we have 
\begin{equation}\label{eq:tilde_YN_control}
\sup_{s\geq 0} \Vert \tilde{Y}_N(s) - Y_N(s)\Vert_2 \xrightarrow[N\to\infty]{}0.
\end{equation}
\end{lem}
\begin{proof} A direct computation gives, for any $s\geq 0$,
\begin{align*}
\Vert \tilde{Y}_N(s) - Y_N(s)\Vert_2^2
&=\sum_j \int_{B_{N,j}} \left( X_{N,j}(s)-X_\infty(x_j)-X_N(s)(y)+X_\infty(y)\right)^2dy.
\end{align*}
By definition of $X_N(s)$ in \eqref{eq:def_UN}, $X_N=X_{N,j}$ on $B_{N,j}$ hence using Theorem \ref{thm:large_time_cvg_u_t}
$$\Vert \tilde{Y}_N(s) - Y_N(s)\Vert_2^2=\sum_j \int_{B_{N,j}} \left(X_\infty(y) -X_\infty(x_j)\right)^2dy.$$ 
Then \eqref{eq:tilde_YN_control} is a straightforward consequence of the uniform continuity of $X_\infty$ on the compact $I$ (see Theorem \ref{thm:large_time_cvg_u_t}). It still holds under the hypotheses of Section \ref{S:extension} by decomposing the sum on each interval $C_k$.
\end{proof}

We will often use
\begin{equation}\label{eq:tilde_YN_maj}
\dfrac{1}{N} \sum_{j=1}^N\left\vert Y_N(s)(x_j)\right\vert^2 =\Vert \tilde{Y}_N(s)\Vert_2^2\leq \dfrac{1}{2} \left(1+\Vert \tilde{Y}_N(s)\Vert_2^4\right)\leq  \dfrac{1}{2} \left(2+\Vert Y_N(s)\Vert_2^4\right),
\end{equation}
the last inequality being true for $N$ large enough (independently of $s$) using Lemma \ref{lem:tilde_YN_control}.

\begin{lem}\label{lem:control_Rnk}
Under Hypothesis \ref{hyp_globales},
\begin{equation}
R^W_{N,k}\xrightarrow[N\to\infty]{}0, \quad k\in \{1,2\}, \quad S^W_{N}\xrightarrow[N\to\infty]{}0,
\end{equation}
where $R_{N,k}^W$ and $S_N^W$ are respectively defined in \eqref{eq:def_Rnk} and \eqref{eq:def_Sn}.
\end{lem}

\begin{proof}
Fix $\varepsilon>0$. As $W$ is uniformly continuous on $I$, there exists $\eta>0$ such that $\vert W(x,y) - W(x,z)\vert \leq \epsilon$ for any $(x,y,z)\in I^3$ with $\vert y-z\vert \leq \eta$. Then, for $N$ large enough (such that $\frac{1}{N}\leq \eta$, we have directly that $R_{N,1}^W\leq \epsilon$ and $R_{N,2}^W\leq \epsilon$ hence the result. We can do the same for $S_N^W$.
\end{proof}

\begin{lem}\label{lem:drift_term_quadratic}
Under Hypothesis \ref{hyp_globales}, for any $t> t_0\geq 0$, 
\begin{equation}\label{eq:maj_phi_N0}
\left\Vert \phi_{N,0}(t) \right\Vert_2\leq C_{F,W} \int_{t_0}^t e^{-\gamma(t-s)} \left( \Vert Y_N(s) \Vert_2^2 + \delta_s + \delta_s^2\right)ds.
\end{equation}
\end{lem}

\begin{proof}[Proof of Lemma \ref{lem:drift_term_quadratic}]
By Proposition \ref{prop:operateur_L} we have $\left\Vert \phi_{N,0}(t)\right\Vert_2\leq  \int_{t_0}^t e^{-\gamma(t-s)} \left\Vert T_W g_N(s)  \right\Vert_2ds.$
As for any $x\in I$,
$\left\vert T_W g_N(s) (x) \right\vert \leq \int_I W(x,y) \left\vert g_N(s)(y)\right\vert dy,
$ and as
\begin{align*}
\vert g_N(s)(y)\vert &\leq \left\Vert \partial_x^2 F \right\Vert_\infty Y_N(t)(y)^2 + \left\Vert \partial_\eta^2 F\right\Vert_\infty \left\vert\eta_t(y)-\eta_\infty(y)\right\vert^2\\
&\quad+2\left\Vert  \partial_{x,\eta}^2 F \right\Vert_\infty \left\vert Y_N(t)(y) \right\vert \left\vert \eta_t(y)-\eta_\infty(y)\right\vert + \left\Vert \partial_\eta F\right\Vert_\infty \left\vert \eta_t(y)-\eta_\infty(y)\right\vert,
\end{align*}
with Hypothesis \ref{hyp_globales} it gives
\begin{align*}
\Vert T_W g_N(s) \Vert_2^2 &= \int_I \left( \int_I W(x,y) g_N(s)(y) dy\right)^2 dx\\
&\leq C_F \int_I \left( \int_I W(x,y) \left( Y_N(s)(y)^2 + \delta_s^2+Y_N(s)(y)\delta_s+\delta_s\right)dy\right)^2 dx\\
&\leq C_{F,W} \left( \Vert Y_N(s) \Vert_2^4 + \Vert Y_N(s) \Vert_2^2 \delta_s^2 + \delta_s^2 + \delta_s^4\right)\\
&\leq C_{F,W}\left( \dfrac{3}{2} \Vert Y_N(s) \Vert_2^4 + \dfrac{3}{2} \delta_s^2 + \delta_s^4\right)
\end{align*}
as $W$ is bounded. We obtain then, as $\sqrt{a+b}\leq \sqrt{a}+\sqrt{b}$,
$$\left\Vert \phi_{N,0}(t) \right\Vert_2\leq C_{F,W} \int_{t_0}^t e^{-\gamma(t-s)} \left( \Vert Y_N(s) \Vert_2^2 + \Vert Y_N(s) \Vert_2 \delta_s + \delta_s + \delta_s^2\right)ds.$$
Then \eqref{eq:maj_phi_N0} follows as $\left\Vert Y_N(s) \right\Vert_2 \leq \dfrac{1}{2} \left(1+\left\Vert Y_N(s)\right\Vert_2^2\right)$ and $\sup_s \delta_s <\infty$.
\end{proof}

\begin{lem}\label{lem:drift_term_phi_N1}
Under Hypotheses \ref{hyp_globales} and \ref{hyp:scenarios}, $\mathbb{P}$-almost surely for $N$ large enough and for any $t> t_0\geq 0$,
\begin{equation}\label{eq:maj_phi_N1}
\Vert \phi_{N,1}(t)\Vert_2\leq C_F \int_{t_0}^t e^{-(t-s)\gamma} \left\Vert Y_N(s)\right\Vert_2^2ds + G_{N,1},
\end{equation}
where  $G_{N,1}=G_{N,1}(\xi)$ is explicit in $N$ and tends to 0 as $N\to\infty	$. Moreover, if we suppose $F$ bounded, we have a better bound
\begin{equation}\label{eq:maj_phi_N1_B}
\sup_{t>0}\left\Vert \phi_{N,1}(t) \right\Vert_2\leq   \dfrac{C_{F}}{\sqrt{N\rho_N^2}}.
\end{equation}
\end{lem}

\begin{proof}[Proof of Lemma \ref{lem:drift_term_phi_N1}]
Proposition \ref{prop:operateur_L} gives that 
\begin{equation}\label{eq:lem_phi_N1_aux}
\Vert \phi_{N,1}(t)\Vert_2\leq K\int_{t_0}^t e^{-(t-s)\gamma} \Vert \gamma_N(s) \Vert_2ds
\end{equation} with
\begin{equation}
\label{eq:gammaN}
\gamma_N(s):=\sum_{i=1}^N \Theta_{i,s,1}\mathbf{1}_{B_{N,i}}=\sum_{i,j=1}^N \dfrac{1}{N\rho_N}  \overline{\xi_{ij}} F(X_{N,j}(s),\eta_s(x_j))\mathbf{1}_{B_{N,i}}.    
\end{equation}
where we have used the notation
\begin{equation}
\label{eq:overline_xi}
\overline{\xi_{ij}} =\xi_{ij}^{(N)}-W_N(x_i,x_j),     
\end{equation}
Forgetting about the term $F(X_{N,j}(s),\eta_s(x_j))$ in \eqref{eq:gammaN}, $\gamma_N$ is essentially an empirical mean of the independent centered variables $ \overline{\xi_{ij}}$ and thus should be small as $N\to\infty$. One difficulty here is that concentration bounds (e.g. Bernstein inequality) for weighted sums such as $\sum_j \overline{\xi_{ij}} u_{i,j}$ (for some deterministic fixed weight $u_{i,j}$) are not directly applicable, as $u_{i,j}=F(X_{N,j}(s),\eta_s(x_j))\mathbf{1}_{B_{N,i}}$ depends in a highly nontrivial way on the variables $\xi_{i,j}^{(N)}$ themselves. A strategy would be to use Grothendieck inequality (see Theorem~\ref{thm:grothendieck}). We refer here to \cite{Coppini2022,Coppini_Lucon_Poquet2022} where the use of such Grothendieck inequality (and extensions) has been implemented in a similar context of interacting diffusions on random graphs. However here, a supplementary difficulty lies in the fact that $F$ need not be bounded (recall that a particular example considered here concerns the linear case where $F(x, \eta)=x + \mu$). Hence the application of Grothendieck inequality is not straightforward when $F$ is unbounded. For this reason, we give below two different controls on $ \gamma_N$: a general one, without assuming that $F$ and a second (sharper) one, when $F$ is bounded (using Grothendieck inequality). In the first case, we get around the difficulty of unboundedness of $F$ by introducing $F(X_\infty(x_j),\eta_\infty(x_j))$ which is bounded, since $X_{\infty}$ is.

First begin with the general control on $\gamma_N$:
we can write
\begin{align}\label{eq:lem_phiN1_aux}
\gamma_N(s)&= \sum_{i,j=1}^N \dfrac{1}{N\rho_N}  \overline{\xi_{ij}} \left(F(X_{N,j}(s),\eta_s(x_j))-F(X_\infty(x_j),\eta_\infty(x_j))\right)\mathbf{1}_{B_{N,i}}\notag\\&\quad+ \sum_{i,j=1}^N \dfrac{1}{N\rho_N}  \overline{\xi_{ij}} F(X_\infty(x_j),\eta_\infty(x_j))\mathbf{1}_{B_{N,i}}=:\gamma_{N,1}(s) + \gamma_{N,2}(s).
\end{align}
Denoting by $\Delta F_j:=F(X_{N,j}(s),\eta_s(x_j))-F(X_\infty(x_j),\eta_\infty(x_j))$, we have, as $\langle \mathbf{1}_{B_{N,i}}, \mathbf{1}_{B_{N,i'}}\rangle= \dfrac{\mathbf{1}_{i=i'}}{N}$ and with  $\displaystyle S_{jj'}:= \dfrac{1}{N}\sum_{i=1}^N \overline{\xi_{ij}}~ \overline{\xi_{ij'}}$, $\displaystyle \left\Vert \gamma_{N,1}(s)\right\Vert_2^2 = \dfrac{1}{N^2\rho_N^2} \sum_{j,j'=1}^N \Delta F_j \Delta F_{j'} \dfrac{1}{N}S_{jj'}$. Define the following quantity $S_{N}^\text{max}:= \sup_{1\leq j\neq j'\leq N} \left\vert S_{jj'} \right\vert$. The purpose of Lemma \ref{lem:maj_S} is exactly to control $S_{N}^\text{max}$, see in particular \eqref{eq:maj_SNMAX}. We have 
\begin{align*}
\left\Vert \gamma_{N,1}(s)\right\Vert_2^2 &= \left(\dfrac{1}{N^2\rho_N^2} \sum_{j\neq j'=1}^N \Delta F_j \Delta F_{j'} \dfrac{S_{jj'}}{S_{N}^\text{max}}\right)S_{N}^\text{max} + \dfrac{1}{N^3\rho_N^2} \sum_{i,j=1}^N \Delta F_j^2 \overline{\xi_{ij}}^2\\
&\leq  S_{N}^\text{max} \left( \dfrac{1}{N\rho_N^2} \sum_{j=1}^N \left\vert \Delta F_j\right\vert^2\right) + \dfrac{1}{N^2\rho_N^2} \sum_{j=1}^N \Delta F_j^2.
\end{align*}
As $\left\vert \Delta F_j \right\Vert \leq \Vert F \Vert_L \left( \left\vert Y_N(s)(x_j)\right\vert + \delta_s\right)$, we obtain as $s\mapsto\delta_s$ is bounded
$$\left\Vert \gamma_{N,1}(s)\right\Vert_2^2 \leq C_F \left( \left\Vert \tilde{Y}_N(s)\right\Vert_2^2 +1\right)\left(\dfrac{S_{N}^\text{max}}{\rho_N^2} +\dfrac{1}{N\rho_N^2}\right),$$
hence as $\left\Vert Y_N(s) \right\Vert_2 \leq \dfrac{1}{2} \left(1+\left\Vert Y_N(s)\right\Vert_2^2\right)$ and using \eqref{eq:tilde_YN_maj},
\begin{equation}\label{eq:lem_phiN1_gamma1}
\left\Vert \gamma_{N,1}(s)\right\Vert_2^2 \leq C_F \left( \left\Vert Y_N(s)\right\Vert_2^4 +1\right)\left(\dfrac{S_{N}^\text{max}}{\rho_N^2} +\dfrac{1}{N\rho_N^2}\right).
\end{equation}

 For the second term of \eqref{eq:lem_phiN1_aux}, we have
\begin{align*}
\Vert \gamma_{N,2}(s)\Vert_2^2 &= \dfrac{1}{N} \sum_{i=1}^N \left( \dfrac{1}{N\rho_N} \sum_{j=1}^N  \overline{\xi_{ij}} F(X_\infty(x_j),\eta_\infty(x_j)) \right)^2\\
&= \dfrac{1}{N^3\rho_N^2} \sum_{i=1}^N \sum_{j,j'=1}^N   \overline{\xi_{ij}}~  \overline{\xi_{ij'}} F(X_\infty(x_j),\eta(x_j)) F(X_\infty(x_{j'}),\eta_\infty(x_{j'})).
\end{align*}
Let $\displaystyle \alpha_{i,j,j'}:=\dfrac{F(X_\infty(x_j),\eta_\infty(x_j)) F(X_\infty(x_{j'}),\eta_\infty(x_{j'}))}{\Vert F(X_\infty, \eta_\infty)\Vert_\infty^2}\in [0,1]$, $\displaystyle R_k:= \sum_{ \substack{i,j,j'=1\\j\neq j'}}^k \alpha_{i,j,j'}  \overline{\xi_{ij}} ~\overline{\xi_{ij'}}$, and $\mathcal{F}_k=\sigma\left( \xi_{ij}, 1 \leq i,j\leq k\right)$ (with respect to $\mathbb{P}$, i.e. the realisation of the graphs). We have then $\displaystyle\Vert \gamma_{N,2}(s)\Vert_2^2 =  \dfrac{C_{F,X_\infty}}{N^3\rho_N^2}  \sum_{i,j=1}^N \alpha_{i,j,j}  \overline{\xi_{ij}}^2 +\dfrac{C_{F,X_\infty}}{N^3\rho_N^2}  R_N \leq \dfrac{C_{F,X_\infty}}{N\rho_N^2} +\dfrac{C_{F,X_\infty}}{N^3\rho_N^2}  R_N$. We show next that $(R_N)$ is a martingale: for any $k\geq 1$ (note that $R_1=0$):
\begin{multline*}
\mathbb{E}\left[ R_{k+1} \vert \mathcal{F}_k\right] =\mathbb{E}\left[ \left.\sum_{ \substack{i,j,j'=1\\j\neq j'}}^{k+1} \alpha_{i,j,j'}  \overline{\xi_{ij}}~ \overline{\xi_{ij'}} \right\vert \mathcal{F}_k\right]
= R_k  + \mathbb{E}\left[ \left.\sum_{\substack{j,j'=1\\j\neq j'}}^{k} \alpha_{k+1,j,j'}  \overline{\xi_{k+1,j}} ~\overline{\xi_{k+1,j'}} \right\vert \mathcal{F}_k\right]\\+ \mathbb{E}\left[ \left.\sum_{i,j'=1}^{k} \alpha_{i,k+1,j'}  \overline{\xi_{i,k+1}}~ \overline{\xi_{ij'}} \right\vert \mathcal{F}_k\right]+
\mathbb{E}\left[ \left.\sum_{i,j=1}^{k} \alpha_{i,j,k+1}  \overline{\xi_{i,k+1}} ~\overline{\xi_{ij}} \right\vert \mathcal{F}_k\right] \\+ \mathbb{E}\left[ \left.\sum_{j=1}^{k} \left(\alpha_{k+1,k+1,j}+\alpha_{k+1,j,k+1}\right)  \overline{\xi_{k+1,k+1}} ~\overline{\xi_{k+1,j}} \right\vert \mathcal{F}_k\right]=R_k,
\end{multline*}
as $\mathbb{E}\left[\overline{\xi_{ij}} ~\overline{\xi_{ij'}} \vert \mathcal{F}_k\right]=0$ if $j\neq j'$ and at least one of the indexes $i,j,j'$ is equal to $k+1$ by independence of the family of random variables $\left(\xi_{ij}\right)_{i,j}$. Similarly, we have
\begin{align*}
\Delta R_k &= R_{k+1}-R_k=\sum_{\substack{j,j'=1\\j\neq j'}}^{k} \alpha_{k+1,j,j'}  \overline{\xi_{k+1,j}} ~\overline{\xi_{k+1,j'}}  +\sum_{\substack{1\leq i\leq k+1\\ 1\leq  j \leq k}} \left(\alpha_{i,j,k+1}+\alpha_{i,k+1,j'} \right)  \overline{\xi_{i,k+1}}~ \overline{\xi_{ij}} .
\end{align*}
As each $\vert\overline{\xi_{i,j}}\vert\leq 1$ and $\vert\alpha_{i,j,k}\vert\leq 1$, it gives $\vert\Delta R_k \vert\leq 3k^2+k$. Theorem \ref{thm:ineg_AZ-HO} gives then that
\begin{align*}
\mathbb{P}\left( \left\vert \dfrac{C_{F,X_\infty}}{N^3\rho_N^2}  R_N \right\vert \geq x \right) &= \mathbb{P}\left( \vert R_N \vert \geq  \dfrac{xN^3\rho_N^2}{C_{F,X_\infty}} \right)\\
&\leq 2 \exp \left( -\dfrac{\left( \dfrac{xN^3\rho_N^2}{C_{F,X_\infty}}\right)^2}{2\sum_{k=1}^N \left(3k^2+k\right)^2} \right)\\
&=  2 \exp \left(- \dfrac{x^2 N^6\rho_N^4}{C_{F,X_\infty}^2P(N)}\right),
\end{align*}
with $\displaystyle P(N)= 2 N (N+1) \left( \dfrac{9}{5}\left(N+\dfrac{1}{2}\right)\left(N^2+N-\dfrac{1}{3}\right)+\dfrac{3N(N+1)}{2}+\dfrac{2N+1}{6}\right) \sim_{N\to\infty} \dfrac{18}{5} N^5$. For the choice $x^2= \dfrac{C_{F,X_\infty}^2P(N)}{N^{6-2\tau}\rho_N^4}$ with $\tau$ in \eqref{eq:dilution}, ($x^2\propto \dfrac{1}{N^{1-2\tau}\rho_N^4}$) it gives
$$ \mathbb{P}\left( \left\vert \dfrac{C_{F,X_\infty}}{N^3}  R_N \right\vert \geq \sqrt{\dfrac{C_{F,X_\infty}^2P(N)}{N^{6-2\tau}\rho_N^4}} \right) \leq 2 \exp \left(- N^{2\tau}\right),$$
which is summable hence by Borel-Cantelli Lemma, there exists $\mathcal{O}\in\mathcal{F}$ such that $\mathbb{P}(\mathcal{O})=1$ and on $\mathcal{O}$, there exists $\widetilde{N}<\infty$ such that if $N\geq \widetilde{N}$, $\left\vert \dfrac{C_{F,X_\infty}}{N^3}  R_N \right\vert \leq \sqrt{\dfrac{C_{F,X_\infty}^2P(N)}{N^{6-2\tau}\rho_N^4}} \propto \dfrac{1}{N^{1/2-\tau}\rho_N^2}$, hence $\mathbb{P}$-a.s. for $N$ large enough
\begin{equation}\label{eq:lem_phiN1_gamma2}
\left\Vert \gamma_{N,2}(s)\right\Vert_2^2 \leq C\left( \dfrac{1}{N\rho_N^2} + \dfrac{1}{N^{1-2\tau}\rho_N^4}\right)
\end{equation}
Coming back to \eqref{eq:lem_phiN1_aux}, combining  \eqref{eq:lem_phiN1_gamma1} and \eqref{eq:lem_phiN1_gamma2} and a control of $S_N^{\text{max}}$ from Lemma \ref{lem:maj_S}, we have $\mathbb{P}$-a.s. for $N$ large enough
$$ \left\Vert \gamma_{N}(s)\right\Vert_2^2 \leq C_F \left( \left\Vert Y_N(s)\right\Vert_2^4 +1\right)\left(\dfrac{1}{N^{1/2-\tau}\rho_N^2} +\dfrac{1}{N\rho_N^2}\right) +  C_F\left( \dfrac{1}{N\rho_N^2} + \dfrac{1}{N^{1-2\tau}\rho_N^4}\right),$$
hence taking the square root and using \eqref{eq:lem_phi_N1_aux}, 
$$\Vert \phi_{N,1}(t)\Vert_2\leq C_F \int_{t_0}^t e^{-(t-s)\gamma} \left\Vert Y_N(s)\right\Vert_2^2ds + G_{N,1},$$
where $G_{N,1}=C_F\left( \dfrac{1}{N\rho_N^2} + \dfrac{1}{N^{1-2\tau}\rho_N^4}+\dfrac{1}{N^{1/2-\tau}\rho_N^2}\right)\to 0$ under Hypothesis \ref{hyp:scenarios}.

\medskip

Let us now turn to the sharper control on $\gamma_N$ defined in \eqref{eq:gammaN} when $F$ is bounded. Coming back to \eqref{eq:lem_phiN1_aux}, we have
\begin{align*}
\Vert \gamma_N(s)\Vert_2^2 &= \int \left( \sum_{i,j=1}^N \dfrac{1}{N\rho_N}\overline{\xi_{ij}} F\left(X_{N,j}(s),\eta_s(x_j)\right)\mathbf{1}_{B_{N,i}}(x)\right)^2 dx\\
&= \dfrac{1}{N} \sum_{i=1}^N \left( \sum_{j=1}^N \dfrac{1}{N\rho_N} \overline{\xi_{ij}} F\left(X_{N,j}(s),\eta_s(x_j)\right) \right)^2\\
&= \dfrac{1}{N^3\rho_N^2} \sum_{i,j,k=1}^N  \overline{\xi_{ij}}~\overline{\xi_{ik}} F\left(X_{N,j}(s),\eta_s(x_j)\right)F\left(X_{N,k}(s),\eta_s(x_k)\right)\\
&= \left(\dfrac{ \Vert F \Vert_\infty}{N\rho_N}\right)^2\dfrac{1}{N} \sum_{j,k=1}^N \alpha_{jk} F_jF_k 
\end{align*}
with $\alpha_{jk}:=\sum_{i=1}^N\overline{\xi_{ij}}~\overline{\xi_{ik}}$ and $F_j:=\dfrac{F\left(X_{N,j}(s),\eta_s(x_j)\right)}{\Vert F \Vert_\infty}$. Grothendieck  inequality (see Theorem \ref{thm:grothendieck}) gives then that there exists $K>0$ such that
\begin{align*}
\Vert \gamma_N(s)\Vert_2^2 &\leq K \dfrac{1}{N}  \left(\dfrac{
\Vert F \Vert_\infty}{N\rho_N}\right)^2  \sup_{s_j,t_k = \pm 1}\sum_{j,k} \alpha_{jk} s_j t_k
\\
&\leq  \dfrac{C_F}{N^3\rho_N^2}   \sup_{s_j,t_k = \pm 1}\sum_{i,j,k=1}^N  \overline{\xi_{ij}}~\overline{\xi_{ik}}s_j t_k.
\end{align*}
Fix some vectors of signs $s=(s_i)_{1\leq i \leq N}$ and $t=(t_j)_{1\leq j \leq N}$. Let $A=\left( \overline{\xi_{ij}} \right)_{1\leq i,j \leq N}$, then $\displaystyle\sum_{i,j,k=1}^N \overline{\xi_{ij}}~\overline{\xi_{ik}}s_j t_k = \langle t, A^*As\rangle$ where $\langle,\rangle$ denotes the scalar product in $\mathbb{R}^N$ and $A^*$  the transpose of $A$.
 As for any sign vector $t$, $\Vert t \Vert^2= \sum_{k=1}^N t_k^2 = N$, and $\Vert A^* A \Vert = \Vert A \Vert_{\text{op}}^2$, we obtain as $\vert \langle t, A^*As\rangle \vert \leq \Vert t \Vert \Vert A^* A s \Vert \leq N \Vert A \Vert_{\text{op}}^2$:
$$ \Vert \gamma_N(s)\Vert_2^2 \leq  \dfrac{C_F}{N^3\rho_N^2}    N \Vert A \Vert_{\text{op}}^2 =\dfrac{C_F}{N^2\rho_N^2}  \Vert A \Vert_{\text{op}}^2.$$
From Theorem \ref{thm:tao2012_upper_tail}, there exist $C_a$ and $C_b$ positive constants such that for any $x\geq C_a$,
$$\mathbb{P} \left( \Vert A \Vert_{\text{op}} > x\sqrt{N} \right) \leq C_a \exp \left( -C_bxN \right).$$
We apply it for $x=C_a$, hence, by Borel-Cantelli Lemma as $\exp(-CN)$ is summable, there exists $\widetilde{\mathcal{O}}\in\mathcal{F}$ such that $\mathbb{P}(\widetilde{\mathcal{O}})=1$ and on $\widetilde{\mathcal{O}}$, there exists $\widetilde{N}<\infty$ such that if $N\geq \widetilde{N}$, $\Vert A \Vert_{\text{op}} \leq C_a\sqrt{N}$. We obtain then that 
$$\Vert \gamma_N(s)\Vert_2^2 \leq \dfrac{C_F}{N\rho_N^2}$$
$\mathbb{P}$-a.s. for $N$ large enough, which concludes the proof in the bounded case with \eqref{eq:lem_phi_N1_aux} as 
$\int_{t_0}^t e^{-(t-s)\gamma}ds\leq \dfrac{1}{\gamma}$.
\end{proof}

\begin{lem}\label{lem:drift_term_phi_N2}
Under Hypothesis \ref{hyp_globales}, for any $t>t_0\geq0$, 
\begin{equation}\label{eq:maj_phi_N2}
\Vert \phi_{N,2}(t)\Vert_2 \leq   C_{F,X_\infty,\eta,W} \int_{t_0}^t e^{-(t-s)\gamma} \left(\left\Vert Y_N(s)\right\Vert_2^2+\delta_s\right)ds + G_{N,2},
\end{equation}
where $G_{N,2}$ is explicit in $N$ and tends to 0 as $N\to\infty	$.
Moreover, if we suppose $F$ bounded, we have
\begin{equation}\label{eq:maj_phi_N2_B}
\Vert \phi_{N,2}(t)\Vert_2 \leq  C \left(\int_{t_0}^t e^{-(t-s)\gamma} \delta_sds + \sqrt{R_{N,2}^W} +\dfrac{1}{N}\right),.
\end{equation}
with $R_{N,2}^W$ defined in \eqref{eq:def_Rnk}.
\end{lem}

\begin{proof}[Proof of Lemma \ref{lem:drift_term_phi_N2}] We have, with $\Theta_{s,i,2}$ defined in \eqref{eq:def_Theta2}, $\Theta_{s,i,2}\leq e_{s,i,1}+e_{s,i,2}+e_{s,i,3}$ with
\begin{align*}
e_{s,i,1}&:= \sum_{j=1}^N \int_{B_{N,j}} \left( W(x_i,x_j)-W(x_i,y)\right) \left( F \left( X_N(s,x_j),\eta_s(x_j)\right) - F \left( X_\infty(x_j),\eta_s(x_j)\right)\right)dy\\
e_{s,i,2}&:= \sum_{j=1}^N \int_{B_{N,j}}\left( W(x_i,x_j)-W(x_i,y)\right)F\left(X_\infty(x_j),\eta_s(x_j)\right)dy\\
e_{s,i,3}&:= \sum_{j=1}^N \int_{B_{N,j}} W(x_i,y)\left( F \left( X_N(s,x_j),\eta_s(x_j)\right) - F \left( X_N(s,x_j),\eta_s(y)\right)\right)dy.
\end{align*}
We upper-bound each term. We have as $F$ is Lipschitz continuous
\begin{align*}
e_{s,i,1}&\leq \sum_{j=1}^N \left\vert F \left( X_N(s,x_j),\eta_s(x_j)\right) - F \left( X_\infty(x_j),\eta_s(x_j)\right)\right\vert \left\vert \int_{B_{N,j}} \left( W(x_i,x_j)-W(x_i,y)\right)  dy\right\vert\\
&\leq  \sum_{j=1}^N \Vert F \Vert_L \left\vert Y_N(s)(x_j)\right\vert \left\vert \int_{B_{N,j}} \left( W(x_i,x_j)-W(x_i,y)\right)  dy\right\vert,
\end{align*}
which is upper-bounded by
$$ C_F \left( \sum_{j=1}^N \left\vert \int_{B_{N,j}} \left( W(x_i,x_j)-W(x_i,y)\right)  dy\right\vert\right)^\frac{1}{2}\left( \sum_{j=1}^N \left\vert Y_N(s)(x_j)\right\vert^2 \left\vert \int_{B_{N,j}} \left( W(x_i,x_j)-W(x_i,y)\right)  dy\right\vert \right)^\frac{1}{2}$$
by discrete Jensen's inequality. We have $N\left\vert \int_{B_{N,j}} \left( W(x_i,x_j)-W(x_i,y)\right)  dy\right\vert \leq C$ as $W$ is bounded, hence
\begin{align*}
e_{s,i,1}&\leq  C_{F,W}\left( \sum_{j=1}^N \left\vert \int_{B_{N,j}} \left( W(x_i,x_j)-W(x_i,y)\right)  dy\right\vert\right)^\frac{1}{2}\left(\dfrac{1}{N} \sum_{j=1}^N \left\vert Y_N(s)(x_j)\right\vert^2  \right)^\frac{1}{2}\\
&\leq C_F \left( \sum_{j=1}^N \left\vert \int_{B_{N,j}} \left( W(x_i,x_j)-W(x_i,y)\right)  dy\right\vert\right)^\frac{1}{2}\left\Vert \tilde{Y}_N(s)\right\Vert_2.
\end{align*}
We have then 
\begin{align*}
\dfrac{1}{N} \sum_{i=1}^N e_{s,i,1}^2 &\leq \dfrac{C_F}{N} \sum_{i=1}^N \left( \sum_{j=1}^N \left\vert \int_{B_{N,j}} \left( W(x_i,x_j)-W(x_i,y)\right)  dy\right\vert\right)\left\Vert \tilde{Y}_N(s)\right\Vert_2^2\\
&\leq C_F R_{N,1}^W\left\Vert \tilde{Y}_N(s)\right\Vert_2^2,
\end{align*}
where $R_{N,1}^W$ is defined in \eqref{eq:def_Rnk}.

For the second term, we have as $x\mapsto \sup_{s}F(X_\infty(x),\eta_s(x))$ is bounded
\begin{align*}
\dfrac{1}{N}\sum_{i=1}^N e_{s,i,2}^2 &= \dfrac{1}{N} \sum_{i=1}^N \left( \sum_{j=1}^N \int_{B_{N,j}} \left( W(x_i,x_j)-W(x_i,y)\right) F\left( X_\infty(x_j),\eta_s(x_j)\right)dy \right)^2\\
&\leq  \dfrac{C_F}{N} \sum_{i=1}^N \sum_{j=1}^N \int_{B_{N,j}} \left\vert W(x_i,x_j)-W(x_i,y)\right\vert^2 dy \leq C_F R_{N,2}^W,
\end{align*}
where $R_{N,2}^W$ is defined in \eqref{eq:def_Rnk}.

For the third term, as $F$ is Lipschitz continuous
\begin{align*}
e_{s,i,3}&\leq \sum_{j=1}^N \int_{B_{N,j}} W(x_i,y) \Vert F \Vert_L \vert \eta_s(x_j)-\eta_s(y)\vert dy\\
&\leq \sum_{j=1}^N \int_{B_{N,j}} W(x_i,y) \Vert F \Vert_L \left(   \vert \eta_s(x_j)-\eta_\infty(x_j)\vert + \vert \eta_\infty(x_j)-\eta_\infty(y)\vert\right) dy\\
&\leq C_{F,X,W} \left( \delta_s + \dfrac{1}{N}\right) .
\end{align*}
We obtain then with \eqref{eq:tilde_YN_maj}
\begin{align*}
\dfrac{1}{N}\sum_{i=1}^N \Theta_{s,i,2}^2 &\leq \dfrac{3}{N} \sum_{i=1}^N \left( e_{s,i,1}^2+e_{s,i,2}^2+e_{s,i,3}^2\right)\\
&\leq C_{F,X_\infty,X,W}  \left(R_{N,1}^W\left( 1 + \left\Vert Y_N(s)\right\Vert_2^4\right)+ R_{N,2}^W + \delta_s^2 +\dfrac{1}{N^2} \right).
\end{align*}
With  \eqref{eq:def_phi_nk} and  Proposition \ref{prop:operateur_L}, $\Vert \phi_{N,2}(t)\Vert_2 \leq \int_{t_0}^t e^{-(t-s)\gamma}\Vert \sum_{i=1}^N\Theta_{s,i,2}\mathbf{1}_{B_{N,i}}\Vert_2ds,$ and as $\Vert \sum_{i=1}^N\Theta_{s,i,2}\mathbf{1}_{B_{N,i}}\Vert_2^{2} = \dfrac{1}{N} \sum_{i=1}^N \Theta_{s,i,2}^2$, the result follows with $$G_{N,2}=\sqrt{ R_{N,1}^W + R_{N,2}^W} +\dfrac{1}{N},$$
and Lemma \ref{lem:control_Rnk}.

When $F$ is bounded, similarly we show that
$$\dfrac{1}{N}\sum_{i=1}^N \Theta_{s,i,2}^2 \leq C_{F,X_\infty,\eta,W}  \left( R_{N,2}^W+ \delta_s^2 +\dfrac{1}{N^2} \right),$$
hence the result.
\end{proof}

\begin{lem}\label{lem:drift_term_phi_N3}
Under Hypothesis \ref{hyp_globales}, for any $t> t_0\geq 0$, 
\begin{equation}\label{eq:maj_phi_N3}
\Vert \phi_{N,3}(t)\Vert_2 \leq C_{F,X_\infty,W}\int_{t_0}^t e^{-(t-s)\gamma} \left(\left\Vert Y_N(s)\right\Vert_2^2+\delta_s\right)ds + G_{N,3},
\end{equation}
where $G_{N,3}$ is explicit in $N$ and tends to 0 as $N\to\infty	$.
Moreover, if we suppose $F$ bounded, we have
\begin{equation}\label{eq:maj_phi_N3_B}
\sup_{t\geq 0} \Vert \phi_{N,3}(t)\Vert_2 \leq  \sqrt{S_N^W},
\end{equation}
where $S_N^W$ is defined in \eqref{eq:def_Sn}.
\end{lem}

\begin{proof}[Proof of Lemma \ref{lem:drift_term_phi_N3}]
We have
\begin{align*}
\Vert \sum_{i=1}^N\Theta_{s,i,3}\mathbf{1}_{B_{N,i}}\Vert_2^{2} &= \int_I \left(  \sum_{i=1}^N\Theta_{s,i,3}(x)\mathbf{1}_{B_{N,i}}(x)\right)^2 dx\\
&= \sum_{i=1}^N \int_{B_{N,i}} \left( \int_I \left(W(x_i,y)-W(x,y)\right)F\left(X_N(s,y),\eta_s(y)\right)dy\right)^2 dx\\
&\leq  \sum_{i=1}^N \int_{B_{N,i}} \left( \int_I \left(W(x_i,y)-W(x,y)\right)^2dy\right)\left(\int_I\left(F\left(X_N(s,y),\eta_s(y)\right) \right)^2dy\right) dx,
\end{align*}
 with Cauchy Schwarz's inequality. We can recognize $S_N^W$ defined in \eqref{eq:def_Sn}, and we have that, as $F$ is Lipschitz continuous and $x\mapsto F\left(X_\infty(y),\eta_\infty(y) \right)$ is bounded, 
 \begin{align*}
 &\int_IF\left(X_N(s,y),\eta_s(y)\right)^2dy\\
  &\leq \int_I\left(F\left(X_N(s,y),\eta_s(y)\right)- F\left(X_\infty(y),\eta_s(y)\right) \right)^2dy + \int_I F\left(X_\infty(y),\eta_s(y) \right)^2dy\\
 &\leq \Vert F \Vert_L^2 \int_I Y_N(s)(y)^2 dy+\Vert F(X_\infty,\eta_\infty)\Vert_\infty^2\leq C_{F,W}\left( \Vert Y_N(s)\Vert_2^2 +1\right)\leq C_{F,W}\left( \Vert Y_N(s)\Vert_2^4 +1\right).
 \end{align*}
As before, \eqref{eq:def_phi_nk} and Proposition \ref{prop:operateur_L} give that $\Vert \phi_{N,3}(t)\Vert_2 \leq \int_{t_0}^t e^{-(t-s)\gamma}\Vert \sum_{i=1}^N\Theta_{s,i,3}\mathbf{1}_{B_{N,i}}\Vert_2ds$ and
$\Vert \sum_{i=1}^N\Theta_{s,i,3}\mathbf{1}_{B_{N,i}}\Vert_2^{2} \leq C_{F,W}S_N^W\left( \Vert Y_N(s)\Vert_2^4 +1\right)$ hence the result with $G_{N,3}= C_{F,W} \sqrt{R_{N,3}}$ and Lemma \ref{lem:control_Rnk}. When $F$ is bounded, we directly have $\Vert \sum_{i=1}^N\Theta_{s,i,3}\mathbf{1}_{B_{N,i}}\Vert_2^{2} \leq S_N^W$ hence \eqref{eq:maj_phi_N3_B}  as 
$\int_{t_0}^t e^{-(t-s)\gamma}ds\leq \dfrac{1}{\gamma}.$
\end{proof}

\subsection{Proof of Proposition \ref{prop:drift_term}}
Proposition \ref{prop:drift_term} is then a direct consequence of \eqref{eq:def_phi_n0} and \eqref{eq:def_phi_nk}, of the controls given by Lemmas \ref{lem:drift_term_quadratic}, \ref{lem:drift_term_phi_N1}, \ref{lem:drift_term_phi_N2} and \ref{lem:drift_term_phi_N3}, with $G_N=G_{N,1}+G_{N,2}+G_{N,3}$, and of Lemma \ref{lem:control_Rnk} to have $G_N\to 0$.


\section{About the finite time behavior}\label{S:finite}

In this section, we prove Proposition \ref{prop:finite_time}.

\subsection{Main technical results}

In the following, we denote by $\widehat{Y}_N(t):=X_N(t)-X_t$.

\begin{proof}[Proof of Proposition \ref{prop:finite_time}]
Let $t\leq T$. Recall the definition of $X_N(t)$ in \eqref{eq:def_UN} and $X_t$ in \eqref{eq:def_utx}. Proceeding exactly as in the proof of Proposition \ref{prop:termes_sys_micros}, and recalling the definition of $M_N(t)$ in \eqref{eq:def_M_N}, we have
\begin{multline*}
d\widehat{Y}_N(t)=-\alpha \widehat{Y}_N(t)dt + dM_N(t) + \sum_{i,j=1}^N \mathbf{1}_{B_{N,i}}\dfrac{w_{ij}}{N} F\left(X_{N,j}(t),\eta_t(x_j)\right)dt - T_W F\left(X_t,\eta_t\right)dt\\
=-\alpha \widehat{Y}_N(t)dt + dM_N(t) + \sum_{k=1}^3 \sum_{i=1}^N \Theta_{t,i,k}\mathbf{1}_{B_{N,i}}dt + T_W\left( F\left(X_{N,j}(t),\eta_t(x_j)\right)-F\left(X_t,\eta_t\right) \right)dt
\end{multline*}
with the notations introduced in \eqref{eq:def_Theta1} -  \eqref{eq:def_Theta3}. It gives then, as $\widehat{Y}_N(0)=0$,
$$\widehat{Y}_N(t)=\int_0^t e^{-\alpha(t-s)}\widehat{r}_N(s)ds + \int_0^t e^{-\alpha(t-s)}dM_N(s)=:\widehat{\phi}_N(t) + \widehat{\zeta}_N(t)$$
with
$$\widehat{r}_N(t)=\sum_{k=1}^4 \sum_{i=1}^N \Theta_{t,i,k}\mathbf{1}_{B_{N,i}} + T_W\left( F\left(X_{N}(t),\eta_t\right)-F\left(X_t,\eta_t\right) \right).$$
Note that we obtain a similar expression as for $Y_N$ in Proposition \ref{prop:termes_sys_micros}, but with $e^{-\alpha t}$ instead of the semi-group $e^{t\mathcal{L}}$. We use then the two following results, similar to Propositions \ref{prop:noise_perturbation} and \ref{prop:drift_term}.

\begin{prop}\label{prop:noise_perturbation_finite} Let $T>0$. Under Hypothesis \ref{hyp_globales}, there exists a constant $C=C(T,F,\Vert \eta\Vert_\infty)>0$ such that $\mathbb{P}$-almost surely for $N$ large enough: 
$$\mathbf{E}\left[\sup_{s\leq T} \Vert \widehat{\zeta}_N(s) \Vert_2\right] \leq \dfrac{C}{\sqrt{N\rho_N}}.$$
\end{prop}

\begin{prop}\label{prop:drift_term_finite}
Under Hypotheses \ref{hyp_globales}  and \ref{hyp:scenarios}, for any $t>0$, 
\begin{align}\label{eq:control_drift_finite}
\Vert \widehat{\phi}_N(t)\Vert_2 &\leq C \left(  \int_0^t e^{-\alpha(t-s)} \Vert \widehat{Y}_N(s) \Vert_2ds + \widehat{G}_{N}\right),
\end{align}
where $\widehat{G}_{N}$ is an explicit quantity  to be found in the proof that tends to 0 as $N\to \infty$.
\end{prop}
Their proofs are postponed to the following subsection. Hence we obtain 
$$\left\Vert \widehat{Y}_N(t)\right\Vert_2 \leq C\left( \widehat{G}_{N} + \left\Vert \widehat{\zeta}_N(t)\right\Vert_2+\int_0^t e^{-\alpha(t-s)}\left\Vert \widehat{Y}_N(s)\right\Vert_2ds\right),$$
which gives with Grönwall lemma
$$\sup_{t\leq T} \left\Vert \widehat{Y}_N(t)\right\Vert_2 \leq C\left( \widehat{G}_{N} + \sup_{t\leq T} \left\Vert \widehat{\zeta}_N(t)\right\Vert_2\right).$$
With Proposition \ref{prop:noise_perturbation_finite}, it leads to
$$\mathbf{E}\left[\sup_{t\leq T}\left\Vert\widehat{Y}_N(t)\right\Vert_2\right]\leq C\left( \widehat{G}_N + \dfrac{1}{\sqrt{N\rho_N}}\right),$$
hence the result \eqref{eq:finite_time} as \eqref{eq:dilution} implies $\dfrac{1}{\sqrt{N\rho_N}}\to 0$ and $\widehat{G}_N\to 0$.
\end{proof}


\subsection{Proofs of Propositions \ref{prop:noise_perturbation_finite} and \ref{prop:drift_term_finite}}

\begin{proof}[Proof of Proposition \ref{prop:noise_perturbation_finite}]
We do as for Proposition \ref{prop:noise_perturbation}, and apply Îto's formula on 
$$\widehat{\zeta}_N(t)=\sum_{j=1}^N \int_0^t \int_0^\infty e^{-\alpha(t-s)}\chi_j(s,z)\tilde{\pi}_j(ds,dz).$$
The term $I_0(t)$ in \eqref{eq:zeta_spatial_def} becomes $-\alpha \int_0^t \left\Vert \widehat{\zeta}_N(s)\right\Vert_2ds$ which is still non-positive. About $I_1(t)$ and $I_2(t)$, the proof remains the same aside from the fact that we now consider $\widehat{\zeta}_N$ instead of $\zeta_N$. 
\end{proof}

To prove \ref{prop:drift_term_finite}, we introduce an auxilliary  quantity as in Lemma \ref{lem:tilde_YN_control}.
\begin{lem}\label{lem:control_hat_YN}
Let $\overline{Y}_N(s)(v):=\widehat{Y}_N(s)\left( \dfrac{\lceil Nv\rceil}{N}\right)$.  Then for any $T\geq 0$
\begin{equation}\label{eq:barYN_control}
\sup_{0\leq s\leq T} \Vert\overline{Y}_N(s) -  \widehat{Y}_N(s)\Vert_2 \xrightarrow[N\to\infty]{}0.
\end{equation}
\end{lem} 

\begin{proof}
It plays the role of $\tilde{Y}_N(s)$ introduced in Lemma \ref{lem:tilde_YN_control}. Similarly at what has been done before, we have
$$\left\Vert \widehat{Y}_N(s)-\overline{Y}_N(s)\right\Vert_2^2 = \sum_{j=1}^N \int_{B_{N,j}} \left( \widehat{Y}_N(s)(y)-\overline{Y}_N(s)(y)\right)^2dy = \sum_{j=1}^N \int_{B_{N,j}} \left(X_s(x_j)-X_s(y)\right)^2dy$$
which tends to 0 by uniform continuity of $X$ on $ [0,T]\times I$. It still holds under the hypotheses of Section \ref{S:extension} by decomposing the sum on each interval $C_k$.
\end{proof}

\begin{proof}[Proof of Proposition \ref{prop:drift_term_finite}]
We divide $\widehat{\phi}$ as in \eqref{eq:def_phi_nk} and study each contribution. About $\widehat{\phi}_{N,0}(t):=\int_0^te^{-\alpha(t-s)} T_W\left( F(X_N(s),\eta_s)-F(X_s,\eta_s)\right)ds$, we have
\begin{align*}
\left\Vert T_W\left( F(X_N(s),\eta_s)-F(X_s,\eta_s)\right)\right\Vert_2^2 &\leq C_{W,F} \left( \int_I \Vert F \Vert_L \left\vert X_N(s)(y)-X_s(y)\right\vert dy\right)^2 \\
&\leq C_{W,F} \left\Vert \widehat{Y}_N(s)\right\Vert_2^2,
\end{align*}
which gives
$$\left\Vert \widehat{\phi}_{N,0}(t)\right\Vert_2\leq C_{W,F} \int_0^te^{-\alpha(t-s)}  \left\Vert \widehat{Y}_N(s)\right\Vert_2ds.$$

About $\widehat{\phi}_{N,1}(t):=\int_0^te^{-\alpha(t-s)} \sum_{i=1}^N \frac{\Theta_{s,i,1}}{N} \mathbf{1}_{B_{N,i}} ds$, we do as in Lemma \ref{lem:drift_term_phi_N1}. Instead of inserting the terms $F(X_\infty(x_j),\eta_\infty(x_j))$ in \eqref{eq:lem_phiN1_aux} we insert the terms $F(X_s(x_j),\eta_s(x_j))$, that is
\begin{multline*}
\gamma_N(s)\leq \sum_{i,j=1}^N \dfrac{1}{N}  \kappa_{N,i} \overline{\xi_{ij}} \left(F(X_{N,j}(s),\eta_s(x_j))-F(X_s(x_j),\eta_s(x_j))\right)\mathbf{1}_{B_{N,i}}\\+ \sum_{i,j=1}^N \dfrac{1}{N}  \kappa_{N,i} \overline{\xi_{ij}} F(X_s(x_j),\eta_s(x_j))\mathbf{1}_{B_{N,i}}=:\widehat{\gamma}_{N,1}(s) + \widehat{\gamma}_{N,2}(s).
\end{multline*}
The treatment of $\widehat{\gamma}_{N,1}$ is similar of $\gamma_{N,1}$: we make $\overline{Y}_N(s)$ appear instead of $\tilde{Y}_N$ and obtain 
$\Vert\widehat{\gamma}_{N,1}(s)\Vert_2^2\leq C_F \left( \left\Vert \widehat{Y}_N(s)\right\Vert_2^2 +1\right)\left(\dfrac{S_{N}^\text{max}}{\rho_N^2} +\dfrac{1}{N\rho_N^2}\right)$ with \eqref{eq:barYN_control}. About $\widehat{\gamma}_{N,2}$, we do as $\gamma_{N,2}$ as $\sup_{t\in [0,T],x\in I} F(X_t(x),\eta_t(x))<\infty$ and obtain that $\mathbb{P}$-almost surely if $N$ is large enough, $\Vert\widehat{\gamma}_{N,2}\Vert_2^2\leq  C\left( \dfrac{1}{N\rho_N^2} + \dfrac{1}{N^{1-2\tau}\rho_N^4}\right)$.  
We have then that, $\mathbb{P}$-almost surely if $N$ is large enough,
$$\left\Vert \widehat{\phi}_{N,1}(t)\right\Vert_2\leq C_F \int_{t_0}^t e^{-\alpha(t-s)} \left\Vert \widehat{Y}_N(s)\right\Vert_2ds + G_{N,1},$$
where $G_{N,1}\to 0$. 

About $\widehat{\phi}_{N,k}(t):=\int_0^te^{-\alpha(t-s)} \sum_{i=1}^N \frac{\Theta_{s,i,k}}{N} \mathbf{1}_{B_{N,i}} ds$ for $k\in \{2,3\}$, we proceed similarly, doing as in Lemmas \ref{lem:drift_term_phi_N2} and \ref{lem:drift_term_phi_N3} but instead of inserting the terms $F(X_\infty(x_j),\eta_\infty(x_j))$ we insert the terms $F(X_s(x_j),\eta_s(x_j))$: then there is no $\delta_s$ terms. We obtain then
$$\Vert \widehat{\phi}_{N,2}(t)\Vert_2 \leq  C \int_{t_0}^t e^{-\alpha(t-s)}\left\Vert \widehat{Y}_N(s)\right\Vert_2ds + G_{N,2},$$
and 
$$\Vert \widehat{\phi}_{N,3}(t)\Vert_2 \leq C\int_{t_0}^t e^{-\alpha(t-s)} \left\Vert Y_N(s)\right\Vert_2ds + G_{N,3},$$
where both $G_{N,2}$ and $G_{N,3}$ tends to 0. Note that we can obtain better bounds when $F$ is bounded. By putting all the terms $\widehat{\phi}_{N,k}$ together, we get \eqref{eq:control_drift_finite}.
\end{proof}

\appendix
\section{Auxiliary results}
\label{S:appendix}
\subsection{Concentration results}

\begin{thm}[Grothendieck's inequality as in \cite{Coppini2022}]\label{thm:grothendieck} Let $\{a_{ij}\}_{i,j=1,\cdots,n}$ be a $n\times n$ real matrix such that for all $s_i$, $t_j\in\{-1,1\}$
$$\sum_{i,j=1}^na_{ij}s_it_j\leq 1.$$
Then, there exists a constant $K_R>0$, such that for every Hilbert space $\left( H, \langle \cdot,\cdot\rangle_H\right)$ and for all $S_i$ and $T_j$ in the unit ball of $H$
$$ \sum_{i,j=1}^n a_{ij}\langle S_i,T_j\rangle_H \leq  K_R.$$
\end{thm}

 \begin{thm}[Azuma–Hoeffding inequality]\label{thm:ineg_AZ-HO}
Let $(M_n)$ be a martingale with $M_0=0$. Assume that for all $1\leq k \leq n$,   $\vert \Delta M_k \vert \leq c_k$ a.s. for some constants $(c_k)$. Then for all $x\geq 0$
\begin{equation}\label{eq:ineg_AZ-HO}
\mathbb{P}\left( \vert M_n \vert \geq x \right) \leq 2 \exp \left( -\dfrac{x^2}{2\sum_{k=1}^n c_k^2} \right).
\end{equation}
\end{thm}

\begin{thm}[Upper tail estimate for iid ensembles, Corollary 2.3.5 of \cite{tao2012}]\label{thm:tao2012_upper_tail} Suppose that $M=(m_{ij})_{1\leq i,j \leq n}$, where $n$ is a (large) integer and the $m_{ij}$ are independent centered random variables uniformly
bounded in magnitude by 1. Then there exist absolute constants $C, c > 0$ such that
$$\mathbb{P} \left( \Vert M \Vert_{op} > x\sqrt{n} \right) \leq C \exp \left( -cxn \right)$$
for any $x\geq C$.
\end{thm}

\begin{lem}
\label{prop:estimees_IC}
Under Hypothesis \ref{hyp:scenarios}, we have $\mathbb{P}$-almost surely if $N$ is large enough: 
\begin{equation}\label{eq:estimees_IC}
\sup_{1 \leq j \leq N}  \left( \sum_{i=1}^N \dfrac{\xi_{ij}^{(N)}}{N\rho_N}\right) \leq 2, \quad \sup_{1 \leq i \leq N}  \left( \sum_{j=1}^N \dfrac{\xi_{ij}^{(N)}}{N\rho_N}\right) \leq 2.
\end{equation}
\end{lem}

\begin{proof}
It is a direct consequence of Corollary 8.2 of a previous work \cite{agathenerine2021multivariate}, in the case $w_N=\rho_N$, $\kappa_N=\frac{1}{\rho_N}$, $W_N(x_i,x_j)=\rho_NW(x_i,x_j)$ with $W$ bounded.
\end{proof}

\begin{lem}\label{lem:maj_S}
Let $N\geq 1$, for $j\neq j'$ in $\llbracket 1, N \rrbracket$, let $\displaystyle S_{jj'}:= \dfrac{1}{N}\sum_{i=1}^N \overline{\xi_{ij}}~ \overline{\xi_{ij'}}$ with $\xi$ defined in Definition \ref{def:espace_proba_bb}, and $S_{N}^\text{max}:= \sup_{1\leq j\neq j'\leq N} \left\vert S_{jj'} \right\vert$. Then,
under Hypothesis \ref{hyp:scenarios}, $\mathbb{P}$-a.s.
\begin{equation}\label{eq:maj_SNMAX}
\limsup_{N\to\infty} S_{N}^\text{max}\leq N^{\tau-\frac{1}{2}}
\end{equation}
where $\tau\in(0,\frac{1}{2})$ comes from Hypothesis \ref{hyp:scenarios}.
\end{lem}
\begin{proof}
When $j$ and $j'$ are fixed and $j\neq j'$, $\left(X_i:= \overline{\xi_{ij}}~ \overline{\xi_{ij'}}\right)_{1\leq i \leq N}$ is a family of independent random variables with $\vert X_i\vert\leq 1$, $\mathbf{E}[X_i]=0$ and $\mathbf{E}[X_i^2]\leq 1$. Bernstein's inequality gives then for any $t>0$
$$\mathbf{P}\left( \left\vert \sum_{i=1}^N \overline{\xi_{ij}}~ \overline{\xi_{ij'}} \right\vert>t\right)\leq 2\exp\left( -\dfrac{1}{2} \dfrac{t^2}{N+\frac{t}{3}}\right)$$
hence for the choice $t=N^{\frac{1}{2}+\tau}$ with $\tau\in (0,\frac{1}{2})$,
$$\mathbf{P}\left( \left\vert \sum_{i=1}^N \overline{\xi_{ij}}~ \overline{\xi_{ij'}} \right\vert>N^{\frac{1}{2}+\tau} \right)\leq 2\exp\left( -\dfrac{1}{2} \dfrac{N^{2\tau}}{1+\frac{1}{3}N^{-\frac{1}{2}+\tau}}\right) \leq 2\exp\left( -\dfrac{1}{4}N^{2\tau}\right)$$
as $1+\frac{1}{3}N^{-\frac{1}{2}+\tau}\leq 2$. With an union bound
$$\mathbf{P}\left( \sup_{j\neq j'} \left\vert S_{jj'} \right\vert > \dfrac{1}{N^{\frac{1}{2}-\tau}}\right) \leq 2N^2 \exp\left( -\dfrac{1}{4} N^{2\tau}\right).$$
We apply then Borel Cantelli's lemma and obtain \eqref{eq:maj_SNMAX}.
\end{proof}

\begin{lem}\label{lem:inegalit_concentration_Y}
Fix $N > 1$ and $\left(Y_l\right)_{l=1,\ldots,n}$ real valued random variables defined on a probability space $\left(\Omega, \mathcal{F}, \mathbb{P}\right)$. Suppose that there exists $\nu>0$ such that, almost surely, for all $l = 1,\ldots, n-1$, $Y_l\leq 1$, $\mathbb{E}\left[Y_{l+1} \left| Y_l \right.\right] = 0$  and $\mathbb{E}\left[Y_{l+1}^2 \left|Y_l\right.\right]\leq \nu$. Then 
$$\mathbb{ P} \left(n^{ -1} (Y_{ 1}+ \ldots+ Y_{ n}) \geq x\right) \leq \exp \left( -n \frac{ x^{ 2}}{ 2\nu} B \left( \frac{ x}{\nu}\right)\right)$$ for all $x \geq 0$, where \begin{equation}\label{eq:def_B(u)}
B(u):= u^{-2}\left( \left( 1+u \right) \log \left( 1+u \right) - u \right).
\end{equation}
\end{lem}
\begin{proof}
A direct application of \cite[Corollary 2.4.7]{zeitouni1998large} gives that $$ \mathbb{ P}\left(n^{ -1} (Y_{ 1}+ \ldots+ Y_{ n}) \geq x\right) \leq \exp \left( -n H \left( \frac{ x+v}{ 1+v} \vert \frac{ v}{ 1+v}\right)\right),$$ where $H(p\vert q):= p \log(p/q) +(1-p) \log((1-p)/(1-q))$ for $p,q\in [0, 1]$. Then, the inequality $ H \left( \frac{ x+v}{ 1+v} \vert \frac{ v}{ 1+v}\right)\geq \frac{ x^{ 2}}{ 2v} B \left( \frac{ x}{ v}\right)$ (see \cite[Exercise 2.4.21]{zeitouni1998large}) gives the result.
\end{proof} 
\begin{cor}\label{cor:ineg_concentration_xi_carre} Let $\left(Z_{ij}\right)_{i,j}$ be a family of independent Bernoulli variables, with $\mathbb{E}[Z_{ij}]=m_{ij}$. Let $(\beta{ij})_{ij}$ be a sequence such that for any $i,j$, $ \beta_{i,j}\in (0,1]$.Then, for all $x\geq 0$
$$\mathbb{P}\left( \dfrac{1}{N^2} \sum_{i,j=1}^{N} \beta_{ij} \left( \left(Z_{ij}-m_{ij}\right)^2 - \mathbb{E}\left(Z_{ij}-m_{ij}\right)^2\right) \geq x\right) \leq \exp\left( -\dfrac{N^2x^2}{2}B(x)\right).$$
 \end{cor}
 
 \begin{proof} 
Fix a bijection $\phi_N:\llbracket 1 , N^2 \rrbracket \to \llbracket 1 , N \rrbracket \times \llbracket 1 , N \rrbracket$. For any $k\in \llbracket 1 , N^2 \rrbracket$ and $(i,j)=\phi_N(k)$, let $R_k=\beta_{ij} \left( \left(Z_{ij}-m_{ij}\right)^2 - \mathbb{E}\left(Z_{ij}-m_{ij}\right)^2\right)$. As the $\left(m_{ij}\right)_{i,j}$ are independent, the family of randon variables $\left(R_k\right)_{1\leq k \leq N^2}$ is also independent.
As $R_k\leq 1$ a.s., $\mathbb{E}\left[R_{k+1}\vert R_k\right]=0$ and $\mathbb{E}\left[R_{k+1}^2\vert R_k\right]\leq 1$, Lemma \ref{lem:inegalit_concentration_Y} implies that for any $x\geq 0$,
$$\mathbb{P}\left( \dfrac{1}{N^2} \sum_{k=1}^{N^2} R_k \geq x\right) \leq \exp\left( -\dfrac{N^2x^2}{2}B(x)\right)$$
where $B$ is defined in \eqref{eq:def_B(u)}. 
 \end{proof}
 
 \subsection{Other technical results}

\begin{lem}\label{lem:op_radius}Let $K$ be a kernel from $I^2 \to \mathbb{R}_+$ such that $\sup_{x\in I}\int_I K(x,y)^2dy <\infty$. Let $T_K:g\mapsto T_Kg:=\left(x\to\int_I K(x,y)dy\right)$ be the operator associated to $K$, that can be defined from $L^2(I)\to L^2(I)$ and from $L^\infty(I)\to L^\infty(I)$. We assume that $T_K^2:L^2(I)\to L^2(I)$ is compact. Then 
$$r_{ 2}(T_K)= r_{ \infty}(T_K).$$
\end{lem}

\begin{proof}
First note that for all $p\geq1$, $ r(T_K^{ p})^{ \frac{ 1}{ p}}= \left( \lim_{ n\to\infty} \left\Vert T_K^{ pn} \right\Vert^{ \frac{ 1}{ n}}\right)^{ \frac{ 1}{ p}}= \lim_{ n\to\infty} \left\Vert T_K^{ pn} \right\Vert^{ \frac{ 1}{ pn}}= r(T)$, so that $r(T_K^{ p})= r(T_K)^{ p}$. Hence $r_{ 2}(T_K^{ 2})= r_{ \infty}(T^{ 2})$ gives $r_{ 2}(T_{ K})= r_{ \infty}(T_{K})$. Let us prove that $r_{ 2}(T_K^{ 2})= r_{ \infty}(T_K^{ 2})$ by proving that they have the same spectrum. To do so, first note that $T_K^{ 2}: L^{ \infty}(I) \to L^{ \infty}(I)$ is compact: consider $\left(f_n\right)_n$ a bounded sequence of $L^\infty(I)$. It is then also bounded in $L^2(I)$, and as $T_K:L^2(I)\to L^2(I)$ is compact, there exists a subsequence $\left(f_{\phi(n)}\right)$ such that $T_Kf_{\phi(n)}$ converges in $L^2(I)$ to a certain $g$. Then for any $x\in I$,
$$ \vert T_K^2 f_{\phi(n)} - Tg \vert (x) \leq \int_I K(x,y) \left| T_Kf_{\phi(n)}(y) - g(y) \right| dy \leq C_K \Vert T_Kf_{\phi(n)}-g\Vert_2 \xrightarrow[n\to\infty]{} 0,$$ thus $T_K^2:L^\infty(I)\to L^\infty(I)$ is compact. Hence, if one denotes by $ \sigma_{ \infty}(T_{ K}^{ 2})$ and $ \sigma_{ 2}(T_{K}^{ 2})$ the corresponding spectrum of $T_{K}^{ 2}$ (in $L^{ \infty}(I)$ and $L^{ 2}(I)$ respectively), we have that each nonzero element of $ \sigma_{ \infty}(T_{K}^{ 2})$ and $ \sigma_{ 2}(T_{ K}^{ 2})$ is an eigenvalue of $T_{K}^{ 2}$: let $\mu \in \sigma_2(T_K^2)\setminus\{0\}$, there exists $g\in L^2(I)$ such that $\mu g = T_K^2g$. As $$\left| T_K^2g(x) \right| = \left| \int_I K(x,y) \int_I K(y,z) g(z) ~ \nu(dz)\nu(dy)\right| \leq C_K\Vert g \Vert_2 <\infty,$$ $g = \frac{1}{\mu}T_K^2g \in L^\infty(I)$ and $\mu \in \sigma_\infty(T_K^2)$. Conversely, let  $\mu \in \sigma_\infty(T_K^2)\setminus\{0\}$, there exists $g \in L^\infty(I)$ such that $\mu g = T_K^2g$. As $L^\infty(I)\subset L^2(I)$, $\mu \in \sigma_2(T_K^2)$. Hence $r_{ 2}(T_{K}^{ 2})= r_{ \infty}(T_{K}^{ 2})$ and \eqref{eq:spectral_radii_equal} follows.
\end{proof}

\begin{lem}[Quadratic Gr\"{o}nwall's lemma]\label{lem:gronwal_quadratic}
Let $f$ be a non-negative function piecewise continuous with finite  number of distinct jumps of size inferior to $\theta$ on $[t_0,T]$, let $g$ be a non-negative continuous function  and $h \in L_1$.. For any $t\in[t_0,T]$, assume $f$ satisfies
$$f(t)\leq  f(t_0)+g(t) + \int_{t_0}^t h(t-s) f(s)^2 ds.$$
Then, for $\delta<\dfrac{1}{9\Vert h\Vert_1}$, if $\theta\leq \dfrac{\delta}{2}$ and if $\sup_{t\in [t_0,T]}g(t) \leq \delta$, 
we have
$$\sup_{t\in [t_0,T]} f(t) \leq  f(t_0)+3\delta.$$
\end{lem}

\begin{proof}
Let $A=\{t\in [t_0,T], f(t)>f(t_0)+3\delta\}$, suppose $A\neq \emptyset$. Let $t^*=\inf\{t\in [t_0,T], f(t)>f(t_0)+3\delta\}$. If there is no jump at $t_0$, by the initial conditions  $t^*>t_0$, and if there is a jump,  $f(t_0^+)\leq f(t_0)+\dfrac{\delta}{2}$ hence we also have $t^*>t_0$. Moreover, for all $t\in [t_0,t^{*-}]$, $f(t)\leq f(t_0)+ \delta+9\delta^2 \int_{t_0}^t h(t-s)ds\leq f(t_0)+2\delta$. If there is a jump at $t^*$, it is of amplitude $\theta\leq\dfrac{\delta}{2}$ hence $f(t^*)\leq f(t_0)+ \dfrac{5\delta}{2}<f(t_0)+3\delta$ which is a contradiction. If there is no jump at $t^*$, by local continuity we have $f(t^*)\leq  f(t_0)+ \delta+9\delta^2 \int_{t_0}^{t^*} h(t-s)ds\leq f(t_0)+2\delta$ which is also a contradiction. We conclude then that $\sup_{t\in [t_0,T]} f(t) \leq  f(t_0)+3\delta$.

\end{proof}

\bibliographystyle{abbrv}

\end{document}